\documentclass[12pt,reqno]{amsart}
\usepackage{amsmath,amssymb,amsthm,amsfonts,enumerate}
\usepackage{bm}
\usepackage{mathrsfs}
\usepackage{xcolor}
\usepackage[colorlinks=true, linkcolor={blue},citecolor={purple},urlcolor={teal}]{hyperref}

\usepackage[left=1.5in,right=1.5in,top=1.in,bottom=1.in]{geometry}

\newcommand{\R}{\mathbb{R}}
\newcommand{\Q}{\mathbb{Q}}
\newcommand{\Z}{\mathbb{Z}}
\newcommand{\N}{\mathbb{N}}
\newcommand{\mL}{\mathcal{L}}

\DeclareRobustCommand{\rchi}{{\mathpalette\irchi\relax}}
\newcommand{\irchi}[2]{\raisebox{\depth}{$#1\chi$}}

\newtheorem{thm}{Theorem}[section]
\newtheorem{lem}[thm]{Lemma}

\newtheorem{cor}[thm]{Corollary}
\theoremstyle{remark}

\newtheorem{defn}{Definition}[section]
\newtheorem{conj}[defn]{Conjecture}
\newtheorem{prob}[defn]{Problem}

\newtheorem{rem}[defn]{Remark}

\numberwithin{equation}{section}

\author{Jing-Jing Huang}
\address{Department of Mathematics and Statistics, University of Nevada, Reno, 1664 N. Virginia St., Reno, NV 89557, USA}
\email{jingjingh@unr.edu}
\dedicatory{Dedicated to Professor John Friedlander}
\thanks{Research is partially supported by the Simons Foundation Mathematics and Physical Sciences-Collaboration Grants for Mathematicians 855860}


\begin{document}

\title[Extremal affine subspaces]{Extremal affine subspaces and Khintchine-Jarn\'{i}k type theorems}

\begin{abstract}
We prove a conjecture of Kleinbock which gives a clear-cut classification of all extremal affine subspaces of $\R^n$. We also give an essentially complete classification of all Khintchine type affine subspaces, except for some boundary cases within two logarithmic scales. More general Jarn\'{i}k type theorems are proved as well, sometimes without the monotonicity of the approximation function.  These results follow as consequences of  our novel estimates for the number of rational points close to an affine subspace in terms of diophantine properties of its defining matrix.  Our main tool is the multidimensional large sieve inequality and its dual form.

 \end{abstract}
\maketitle

\tableofcontents

\section{Diophantine exponents and extremal manifolds}\label{DioExpExtremal}

\subsection{Dirichlet's theorem and diophantine exponents of matrices}

Classical diophantine approximation concerns how well one can approximate points in $\R^n$ by the rationals in terms of the heights of the latter. For example, Dirichlet's theorem states that for any $\bm{y}=(y_1,y_2,\ldots, y_n)$, one can find infinitely many $(q,p_1,\ldots, p_n)\in\N\times\Z^{n}$ such that 
$$
\left|\bm{y}-\frac{\bm{p}}q\right|<q^{-1-\frac1n}.
$$
Here and throughout the paper, the absolute value $|\cdot|$ will be used to denote the supremum norm of  vectors or matrices, e.g.
$$
|\bm{y}|:=|\bm{y}|_\infty=\max\{|y_1,|y_2|,\ldots, |y_n|\}.
$$
The drop of the subscript $\infty$ should cause no confusion since we will not make use of the Euclidean norm under any circumstance. Also bold letters are used exclusively for vectors, and by default they are assumed to be row vectors.

In a more abbreviated form, Dirichlet's theorem can be rephrased as 
$$
\mathcal{S}_n(1/n)=\R^n
$$
where
\begin{equation}\label{TauAppr}
\mathcal{S}_n(\tau):=\left\{\bm{y}\in\R^n: \|{q}\bm{y}\|<q^{-\tau} \textrm{ for infinitely many }{q}\in\N\right\}
\end{equation}
denotes the set of points that can be simultaneously approximated by infinitely many rational points $\bm{p}/q$ within distance $q^{-1-\tau}$,
 and  as usual 
$$\| \bm{y}\|:=\min_{\bm{k}\in\Z^n}|\bm{y}-\bm{k}|$$ means the $\ell^\infty$ distance of $\bm{y}$ to the nearest integer vector.

The same discussion can be extended to the setting of matrices.   We will first introduce the concept of diophantine exponent of a matrix.

\begin{defn}
The \emph{diophantine exponent} $\omega(A)$ of a $d\times m$ real matrix $A$ is defined as the supremum of the set of $\omega\in\R$ such that
$$
\|A\bm{q}^\mathrm{T}\|<|\bm{q}|^{-\omega}
$$
 holds for infinitely many $\bm{q}\in\Z^m$.
In particular, if $\omega(A)<\infty$, we say $A$ is \emph{diophantine}. \end{defn}

It is a straightforward consequence of Minkowski's theorem on linear forms that 
\begin{equation}\label{DirichletExponent}
\omega(A)\ge\frac{m}d
\end{equation}
for all matrices $A\in\R^{d\times m}$.  This can also be proved directly via Dirichlet's box principle, and for this reason $m/d$ is sometimes called the Dirichlet exponent in $\R^{d\times m}$. Moreover, a Borel-Cantelli argument reveals that 
\begin{equation}\label{GenericExponent}
\omega(A)=\frac{m}d\quad\text{for almost all }A\in\R^{d\times m}
\end{equation}
in the sense of Lebesgue measure on $\R^{d\times m}$.

Within the above context, apparently $\R^n$ can be realized either as $\R^{n\times1}$ (column vectors) or $\R^{1\times n}$ (row vectors). The former case is usually referred to as \emph{simultaneous diophantine approximation}, while the latter as  \emph{dual diophantine approximation}. In particular, for $\bm{y}\in\R^n$, we have two diophantine exponents: $\omega(\bm{y})$ and $\omega(\bm{y}^\mathrm{T})$. Throughout the paper, we will use $\sigma(\bm{y})$ to represent the latter for simplicity.  So when $\bm{y}$ is identified with a row vector (as is the case in this paper), $\omega(\bm{y})$ is called the \emph{dual diophantine exponent} of $\bm{y}$ and $\sigma(\bm{y})$ is called the \emph{simultaneous diophantine exponent} of $\bm{y}$. The Dirichlet bound \eqref{DirichletExponent} then states in this case that $\sigma(\bm{y})\ge1/n$ and $\omega(\bm{y})\ge n$. 

Though the two diophantine exponents defined above are linked via Khintchine's transference principle (see \cite[Chapter V Theorem 4]{Cas57}) 
\begin{equation}\label{TransferPrinciple}
\frac{\omega(\bm{y})}{(n-1)\omega(\bm{y})+n}\le \sigma(\bm{y})\le\frac{\omega(\bm{y})-n+1}{n},
\end{equation}
knowing one of them does not in general allow us to determine the other, except for the special case that they are respectively equal to the corresponding Dirichlet exponent, i.e.  $\omega(\bm{y})=n\iff\sigma(\bm{y})=1/n$.

It is clear that $\mathcal{S}_n(\tau)$ and $\sigma(\bm{y})$ are  related in that 
\begin{equation}\label{SigmaSRelation}
\bigcup_{\tau>\sigma}\mathcal{S}_n(\tau)=\big\{\bm{y}\in\R^n\bigm|\sigma(\bm{y})>\sigma\big\}.
\end{equation}
As mentioned above, it follows from the Borel-Cantelli lemma that 
$|\mathcal{S}_n(\tau)|_{\R^n}=0$ whenever $\tau>1/n$, 
where $|\cdot|_{\R^n}$ stands for the Lebesgue measure on $\R^n$.
 This means that the simultaneous Dirichlet exponent $1/n$ is best possible  for almost all points in $\R^n$.
However, the issue becomes highly nontrivial if one focuses on a given submanifold $\mathcal{M}$ of $\R^n$ and asks whether the Dirichlet exponent $1/n$ still remains best possible for almost all points on $\mathcal{M}$. This naturally leads to the definition of extremal manifolds. 

\subsection{Extremal manifolds}\label{ExtremalManifolds}
\begin{defn}
A submanifold $\mathcal{M}\subseteq\R^n$  is said to be \emph{extremal} if $$|\mathcal{S}_n(\tau)\cap\mathcal{M}|_{\mathcal{M}}=0 \quad\text{when}\;\;\tau>1/n,$$ where $|\cdot|_\mathcal{M}$ means the induced Lebesgue measure on $\mathcal{M}$.
\end{defn}

Note that the notion of extremality is very much extrinsic in nature, as it heavily depends on how the manifold is embedded (or more generally immersed) in $\R^n$. It follows from the discussion above that $\R^n$ itself, when viewed as a manifold in $\R^n$, is extremal. 

In 1964, Sprind\v{z}uk proved that the Veronese curve 
$$
\mathcal{V}_n=\left\{(x,x^2,\ldots, x^n)\bigm|x\in\R\right\}
$$
is extremal \cite{Spr69}, settling a conjecture of Mahler. Sprind\v{z}uk then went on to conjecture that any analytic nondegenerate manifold is extremal. For analytic manifolds, ``nondegenerate" just means ``not contained in any hyperplanes". We take this opportunity to  refer the readers to the excellent book of Sprind\v{z}uk \cite{Spr79} for more backgrounds surrounding this problem and a very nice introduction to metric diophantine approximation in general.     After decades of partial progress by numerous people, Sprind\v{z}uk's far-reaching conjecture was eventually proved in full generality in a milestone work by Kleinbock and Margulis in 1998 using homogeneous dynamics \cite{KM98}.

On the other side of the spectrum, there are many degenerate manifolds, such as those contained in proper affine subspaces of $\R^n$. Further developing the method in \cite{KM98}, Kleinbock has subsequently proved that an analytic manifold is extremal if and only if its affine span (the minimal affine subspace that contains it) is extremal \cite{Kle03}. In other words, this remarkable result says that the extremality of an affine subspace is inherited by its nondegenerate submanifolds. Therefore, to classify extremal manifolds, it is imperative to understand exactly what affine subspaces are extremal.

It is not hard to see that affine subspaces  may or may not be extremal. For instance, a proper $d$-dimensional rational affine subspace of $\R^n$ (images of $\R^d$ under some rational affine transformation) exhibits similar diophantine properties to those of $\R^d$ rather than $\R^n$, in that a typical point on it has simultaneous diophantine exponent $1/d$, which is greater than the Dirichlet exponent $1/n$ in $\R^n$. As a result, proper rational affine subspaces are clearly not extremal. Next, we will discuss some known extremal affine subspaces in the literature before unfolding our first main result. 

The first positive example seems to be due to Schmidt \cite{Sch64}, who proved in 1964 that  a line in $\R^n$ parametrized as follows
\begin{equation}\label{LineInRn}
\mathcal{L}_{\bm{\alpha},\bm{\beta}}=\left\{(x,\alpha_1x+\beta_1,\ldots, \alpha_{n-1}x+\beta_{n-1})\bigm|x\in\R\right\}
\end{equation}
is extremal if either $\omega(\bm{\alpha})$ or $\omega(\bm{\beta})$ is equal to $n-1$. Even though Schmidt's condition is typical in view of \eqref{GenericExponent}, it turns out that it can be greatly relaxed as we shall see soon. 
In the homogeneous case $\bm{\beta}=\bm{0}$, that is, $\mathcal{L}$ is a line passing through the origin, more has been proved.  It follows from independent works of Beresnevich, Bernik, Dickinson, Dodson \cite{BBDD00}, and Kovalevskaya \cite{Kov00}, that $\omega(\bm{\alpha})<n$ suffices for $\mathcal{L}_{\bm{\alpha},\bm{0}}$ to be extremal. Soon after, Kleinbock made a systematic study of this problem in general \cite{Kle03,Kle08},  and in order to properly state his main results we need to introduce some notation.  

 Consider a $d$-dimensional affine subspace $\mathcal{L}$ of $\R^n$ parametrized as follows

 \begin{equation}\label{AffineSubspace}
        \mathcal{L}=\mathcal{L}_A :=\left\{(\bm{x},\widetilde{\bm{x}}A)\bigm|\bm{x}\in{\R^d} \right\},
\end{equation}
where
$$
\widetilde{\bm{x}}=(1,\bm{x}),
$$ 
$$
A=\begin{pmatrix}
\bm{\beta}\\
A'
\end{pmatrix}
\in \R^{(d+1)\times m},\quad A'\in \R^{d\times m},
$$
and
$$
d=\dim\mathcal{L},\quad m=n-d=\text{codim}\,\mathcal{L}.
$$
Kleinbock has shown that if all rows (resp. columns) of $A$ are rational multiples of one row (resp. column), then $\mathcal{L}$ is extremal if and only if $\omega(A)\le n$ (see \cite[Theorem 1.3]{Kle03} and its strengthening \cite[Theorem 0.4]{Kle08}). In particular, the following scenarios are covered by his theorem. 
\begin{enumerate}[(K1)]
\item {affine coordinate subspaces: } $A=\begin{pmatrix}\bm{\beta}\\\bm{0}\end{pmatrix}\in\R^{(d+1)\times m}$ with $\bm{\beta}\in\R^m$;\\
\item lines passing through the origin: $A=\begin{pmatrix}\bm{0}\\\bm{\alpha}\end{pmatrix}\in\R^{2\times(n-1)}$ with $\bm{\alpha}\in\R^{n-1}$;\\
\item hyperplanes: $A=(\beta,\alpha_1,\ldots, \alpha_{n-1})^\mathrm{T}\in\R^{n\times 1}$.\\
\end{enumerate}

Clearly, the assumption in Kleinbock's theorem requires that the rank of $A$ must be 1, which is rather restrictive. For this reason
Kleinbock went on to remark that 
\begin{quote}
    ``\emph{it would be very interesting to find out whether [his theorem] can be extended to the cases when the rank of $A$ is greater than one}."
\end{quote}
Here in our first main result, we answer Kleinbock's question affirmatively for all affine subspaces. 
\begin{thm} \label{ExtremalCondition}
Let $\mathcal{L}$ be an affine subspace parametrized by $A$ as in \eqref{AffineSubspace}. Then we have the following equivalence:
$$
\mathcal{L}\text{ is extremal}\iff
\omega(A)\le n.
$$
\end{thm}

Theorem \ref{ExtremalCondition} gives a clear-cut criterion for an affine subspace to be extremal. Therefore, when combining it with Kleinbock's inheritance principle discussed above,  we have an explicit and complete classification of all extremal  manifolds in $\R^n$.

\begin{rem}
It is interesting to note that Theorem \ref{ExtremalCondition} even holds for zero-dimensional affine subspaces. In this case,   $\mathcal{L}$ is simply a point, and we may think that the parametrization matrix $A$ of $\mathcal{L}$ is given by a $1\times n$ matrix/row vector $\bm{\beta}\in\R^n$ which represents the coordinates of this point. As the extremality of $\mathcal{L}=\{\bm{\beta}\}$ is then defined with respect to the delta measure supported at $\bm{\beta}$, to say $\mathcal{L}=\{\bm{\beta}\}$ is extremal just means exactly $\sigma(\bm{\beta})=1/n$. On the other hand, note that the condition $\omega(\bm{\beta})\le n$ actually becomes $\omega(\bm{\beta})= n$ since the reverse inequality is true by the Dirichlet bound \eqref{DirichletExponent}. So in this case, Theorem \ref{ExtremalCondition} simply reduces to the statement that $\sigma(\bm{\beta})=1/n$ if and only if $\omega(\bm{\beta})=n$, which is a well known consequence of Khintchine's transference principle  \eqref{TransferPrinciple}. 
\end{rem}

\begin{rem}
Equipped with Theorem \ref{ExtremalCondition}, we may revisit Schmidt's example \eqref{LineInRn} and confirm that the correct extremal criterion there should be $\omega{\bm{\beta}\choose\bm{\alpha}}\le n$, since the parametrization matrix $A$ for the line in \eqref{LineInRn} is exactly ${\bm{\beta}\choose\bm{\alpha}}\in\R^{2\times(n-1)}$. Note that by the definition of diophantine exponent $\omega{\bm{\beta}\choose\bm{\alpha}}\le \min\{\omega(\bm{\alpha}),\omega(\bm{\beta})\}$, and the typical value of $\omega{\bm{\beta}\choose\bm{\alpha}}$ is $\frac{n-1}2$. So we have indeed greatly relaxed Schmidt's condition as promised. 
\end{rem}

It has been observed by Kleinbock that the forward implication in Theorem \ref{ExtremalCondition} can be proved by a short elementary argument \cite[Lemma 5.4]{Kle08}. Therefore it is the reverse implication there that consists the main substance of the theorem. Actually in Kleinbock's dynamical approach, some ``higher order" exponents $\omega_j(A)$, $j=1,2,\ldots, m$, come up very naturally. In particular, the first one $\omega_1(A)$ of those coincides with the ``regular" exponent $\omega(A)$ that we have defined earlier. With these higher order exponents in hand, Kleinbock has established the following criterion 
\begin{equation}\label{KleinbockCondition}
\mathcal{L}\text{ is extremal}\iff
\max_{1\le j\le m}\omega_j(A)\le n.
\end{equation}
Since the definition of these higher order exponents are rather technical and they are not used in our proofs, we refrain ourselves from reproducing them here for conciseness and refer the interested readers to Kleinbock's papers \cite{Kle03,Kle08} for more details about them.   In the codimension one case (K3), Kleinbock's criterion \eqref{KleinbockCondition} just becomes $\omega(A)\le n$ since there is only one exponent to start with.   In the other two cases (K1) and (K2) or more generally when all rows of $A$ are rational multiples of one row, Kleinbock relates them to zero dimensional affine subspaces (namely points) by showing that $\omega_j(A)$ is invariant under a certain row reduction process, and then verifies directly that $\omega_j(A)\le \omega(A)$ when $A$ is a row vector (i.e. $\mL_A$ is a point). In other words, Kleinbock has proved that for those special cases the higher order exponents do not play any role in the maximum on the right hand side of \eqref{KleinbockCondition} and therefore are redundant. Indeed, he has speculated that the sole condition $\omega(A)\le n$ might be enough to prove the reverse implication.  Needless to say, these higher order exponents are very difficult to compute in the general case, as evidenced by some more comments of Kleinbock:
\begin{quote}
    ``\emph{For matrices with no rational dependence between rows or columns the exponents of orders higher than 1 seem to be hard to understand}."
\end{quote}
\begin{quote}
    ``\emph{Even when $\det(A)=0$ $($that is, rows$/$columns of $A$ are linearly dependent over $\R$ but not over $\Q)$, the situation does not seem to be any less complicated}."
\end{quote}
Apparently, it is the presence of those higher order exponents that represents the main hurdle to reach Theorem \ref{ExtremalCondition} via Kleinbock's approach.

Now from a different perspective, we may combine Theorem \ref{ExtremalCondition} with Kleinbock's criterion \eqref{KleinbockCondition} to deduce the following corollary.

\begin{cor}\label{HigherOrderExp} Let $A\in\R^{(d+1)\times m}$ and $n=d+m$. 
If $\omega(A)\le n$, then $\omega_j(A)\le n$ for all $1\le j\le m$. 
\end{cor}
This corollary is a purely algebraic fact, since the definition of $\omega_j(A)$ involves some complicated multilinear algebra and a priori has nothing to do with extremal manifolds. Interestingly, we prove this result through the two extremality criteria of the affine subspace parametrized by $A$, which are established by totally different methods. As mentioned previously, Kleinbock's criterion \eqref{KleinbockCondition} is proved by dynamical methods, while we will see later that Theorem \ref{ExtremalCondition} is eventually proved by Fourier analytic methods.

Recall that $\omega(A)=\frac{m}{d+1}$ for almost all $A\in\R^{(d+1)\times m}$, so the condition $\omega(A)\le n$ in Theorem \ref{ExtremalCondition} is highly typical. Actually much more is true, as we may use Theorem \ref{ExtremalCondition} to compute the Hausdorff dimension of the set of non-extremal affine subspaces with the help of a formula of Bovey and Dodson \cite{BD86}
$$
\dim\left\{A\in\R^{(d+1)\times m}\bigm|\omega(A)>n\right\}=(d+1)m-d.
$$
\begin{cor}\label{NonExtremalDim}
The set of  $d$-dimensional non-extremal affine subspaces of $\R^n$ has Hausdorff dimension $(d+1)m-d$.
\end{cor} 

Since the space of all $d$-dimensional affine subspaces of $\R^n$, which can be identified with the Grassmannian $\mathbf{Gr}(d+1,n+1)$ of $(d+1)$-dimensional linear subspaces of $\R^{n+1}$ through homogenization, has dimension $(d+1)m$,   Corollary \ref{NonExtremalDim} simply infers that the non-extremal ones form a subset of codimension $d$ in this Grassmanian manifold.  
\begin{rem}
Observe that $\mathcal{L}$ is contained in a rational hyperplane if and only if $\|A\bm{q}^\mathrm{T}\|=0$ for some nonzero $\bm{q}\in\Z^m$. Clearly such an $\mathcal{L}$ must not be extremal since $\omega(A)=\infty$. It is then interesting to note that the set of such affine subspaces has Hausdorff dimension $(d+1)(m-1)$, which is exactly one less than $(d+1)m-d$. 
\end{rem}

\begin{rem}
After the completion of this paper, it was pointed out by Dmitry Kleinbock to us that Theorem \ref{ExtremalCondition} had been proved by David Simmons in an unpublished note. While his work is not available publicly, a similar argument has appeared in \cite[Lemma 3.4]{SY}. We are thankful to Dmitry Kleinbock for the communication in this regard and for sending us Simmons' note. 
Simmons has proved exactly that the higher order exponents in \eqref{KleinbockCondition} are redundant as speculated by Kleinbock. Our approach to Theorem \ref{ExtremalCondition} is completely different and actually we deduce Theorem \ref{ExtremalCondition} as a very special case of Theorem \ref{DiophantineExp} below, which is further derived from our novel counting results about the number of rational points close to an affine subspace (see \S\ref{CountingResults}).
\end{rem}

\subsection{Diophantine exponents of affine subspaces}
For extremal submanifolds of $\R^n$, almost all points have simultaneous diophantine exponent $1/n$. In general,
for a submanifold $\mathcal{M}$ of $\R^n$, we shall define its simultaneous diophantine exponent with a fixed inhomogeneous shift. To that end,  
for any $\bm{\theta}\in\R^n$, we may extend \eqref{TauAppr} to define the set of inhomogeneously $\psi$-approximable points as follows
\begin{equation}\label{InhomoPsiAppr}
\mathcal{S}_n(\psi,\bm{\theta}):=\left\{\bm{y}\in\R^n: \|{q}\bm{y}-\bm{\theta}\|<\psi(q) \textrm{ for infinitely many }{q}\in\N\right\}
\end{equation}
where  $\psi:\N\rightarrow\R_{\ge0}$ is going to be called an \emph{approximation function}. In the homogeneous case $\bm{\theta}=\bm{0}$ we simply write $\mathcal{S}_n(\psi)$ for $\mathcal{S}_n(\psi,\bm{0})$. As usual, when $\psi(q)=q^{-\tau}$ we also write $\mathcal{S}_n(\tau,\bm{\theta})$ for $\mathcal{S}_n(\psi,\bm{\theta})$. This abuse of notation should cause no ambiguity. 
Then we consider the inhomogeneous exponent
$$
\sigma(\mathcal{M},\bm{\theta}):=\sup\left\{\tau\in\R:|\mathcal{S}_n(\tau,\bm{\theta})\cap\mathcal{M}|_\mathcal{M}>0\right\},
$$
and in particular when $\bm{\theta}=\bm{0}$ we denote $
\sigma(\mathcal{M}):=\sigma(\mathcal{M},\bm{0})$. With this definition in mind, clearly $\mathcal{M}$ being extremal is equivalent to $\sigma(\mathcal{M})=1/n$. 

In 2009, Y. Zhang \cite{Zha09} proved that for any affine subspace $\mathcal{L}$ and its nondegenerate submanifold $\mathcal{M}$
$$
\sigma(\mathcal{M})=\sigma(\mathcal{L})=\inf\left\{\sigma(\bm{y})\mid\bm{y}\in\mathcal{L}\right\}=\inf\left\{\sigma(\bm{y})\mid\bm{y}\in\mathcal{M}\right\}.
$$
The information this formula conveys is twofold. First of all, it says that nondegenerate submanifolds inherit the exponent from their ``parent" affine subspaces. Secondly,   the typical behavior of $\sigma(\bm{y})$ for $\bm{y}\in\mathcal{M}$ actually coincides with the worst case scenario (the least value).  Zhang has also found explicit formulae for $\sigma(\mathcal{L})$ for the following two cases:
\begin{enumerate}[(Z1)]
    \item 
 $\mathcal{L}$ is a hyperplane ($A$ is a $n\times1$ matrix):
$$
\sigma(\mathcal{L})=\max\left\{\frac1n,\frac{\omega(A)}{n+(n-1)\,\omega(A)}\right\};
$$    
    
    \item
    $\mathcal{L}$ is a line through the origin in $\R^3$ ($A$ is a $2\times2$ matrix with the first row zero):
$$
\sigma(\mathcal{L})=\max\left\{\frac13,\frac{\sigma(A)}{2+\sigma(A)},\frac{\omega(A)}{3+2\,\omega(A)}\right\},
$$
where $\sigma(A):=\omega(A^\mathrm{T})$.
\end{enumerate}

In the theorem below we are able to obtain a uniform upper bound for all inhomogeneous exponents $\sigma(\mathcal{L},\bm{\theta})$ of an arbitrary affine subspace $\mathcal{L}$.

\begin{thm}\label{DiophantineExp}
Let $\mathcal{L}$ be an affine subspace parametrized by $A$ as in \eqref{AffineSubspace}, and suppose that $\omega(A)=\omega$ and $\sigma(A)=\sigma$. Then for any $\bm{\theta}\in\R^n$ we have
\begin{equation}\label{SigmaUpperBound}
\sigma(\mathcal{L,\bm{\theta}})\le \min\left\{ \frac1d\left(1-\frac{m}{\max\{n,\omega\}}\right),\frac{m\sigma-d}n\right\}.
\end{equation}
In particular, this implies the equivalence
 \begin{equation}\label{ExponentEquivalence}
 \omega(A)\le n \iff
\sigma(\mathcal{L,\bm{\theta}})=\frac1n \;\;\text{for all }\bm{\theta}.
\end{equation}
\end{thm}

 Let us explain how the equivalence \eqref{ExponentEquivalence} in Theorem \ref{DiophantineExp}, of which Theorem \ref{ExtremalCondition} is a special case ($\bm{\theta}=\bm{0}$), follows from the upper bound \eqref{SigmaUpperBound} on $\sigma(\mL,\bm{\theta})$. 
 Note that
in the homogeneous case, one also has a lower bound 
\begin{equation}\label{SigmaLowerBound}
 \sigma(\mathcal{L})\ge\max\left\{\frac1n,\frac{\omega(A)}{n+(n-1)\,\omega(A)}\right\}.   
\end{equation}
The first lower bound $1/n$ is obviously due to Dirichlet's theorem, while the second lower bound in \eqref{SigmaLowerBound} is a special case of our Lemma \ref{WellAppr} below (see also \cite[Corollary 15 and Theorem 16]{Zha09} for a different proof). In particular, \eqref{SigmaLowerBound} implies that if $\omega(A)>n$ then $\sigma(\mL)>1/n$. On the other hand, the first upper bound in \eqref{SigmaUpperBound} implies that if $\omega(A)\le n$ then $\sigma(\mL,\bm{\theta})\le\frac1d(1-\frac{m}n)= \frac1n$. So we arrive at the equivalence
$$
\omega(A)\le n\iff \sigma(\mL)=1/n,
$$
which is nothing but Theorem \ref{ExtremalCondition}. Moreover the equivalence \eqref{ExponentEquivalence} immediately follows on invoking an inhomogeneous transference principle of Beresnevich and Velani \cite[Theorem 4(A)]{BV10} which states in our situation that
$$
\sigma(\mathcal{L})=1/n\iff \sigma(\mathcal{L,\bm{\theta}})=1/n \;\;\text{for all }\bm{\theta}.
$$

It is worth noting that our argument actually gives an independent proof of the above inhomogeneous transference principle in the case of affine subspaces. 
To see that  $ \sigma(\mathcal{L,\bm{\theta}})=1/n$ follows from $\sigma(\mathcal{L})=1/n$ (equivalently  $\omega(A)\le n$), we actually only need the lower bound $\sigma(\mathcal{L,\bm{\theta}})\ge1/n$ from an earlier work of Bugeaud and Laurent \cite{BL05} which is also used in \cite{BV10}, as the other direction $\sigma(\mathcal{L,\bm{\theta}})\le1/n$ is already a consequence of our upper bound \eqref{SigmaUpperBound} as mentioned above.

So the main substance of Theorem \ref{DiophantineExp} is the upper bound \eqref{SigmaUpperBound}, whose proof will be taken up in \S\ref{ProofDiophantineExp}.

\begin{rem} 
Since the the exponent of a point on $\mathcal{L}$ should be no smaller than that of  the point in $\R^d$ lying down below, it is clear that $\sigma(\mathcal{L})\le1/d$, which is reflected in Theorem \ref{DiophantineExp}.
It is also interesting to ask when the equality $\sigma(\mL)=1/d$ occurs. In the hyperplane example (Z1) above,  this happens exactly when $\omega(A)=\infty$; but in (Z2), the condition instead becomes $\sigma(A)=\infty$. It is not clear what the sufficient and necessary condition for $\sigma(\mL)=1/d$ is  in general. 
\end{rem}

\begin{rem}\label{SigmaFormula}
The precise behavior of $\sigma(\mathcal{L})$ for general affine subspaces remains elusive. The author unfoundedly suspects that   $\sigma(\mathcal{L})$ is always defined by $m+1$ pieces and that it likely would involve exponents other than $\omega(A)$ and $\sigma(A)$ when $m>2$.
\end{rem}

\section{Khintchine-Jarn\'{i}k type theorems}\label{KhintchineJarnik}

\subsection{Khintchine type manifolds}
 A well known theorem of Khintchine \cite{Khi26} provides a beautiful zero-one law for the set of $\psi$-approximable points $\mathcal{S}_n(\psi)$ with any decreasing (meaning non-increasing in this paper) approximation function $\psi$: almost no/all points in $\R^n$ are $\psi$-approximable in the sense of Lebesgue measure whenever the series $\sum \psi(q)^n$ converges/diverges. Motivated by this theorem and the definition of extremal manifolds, one may as well study Khintchine type manifolds (a terminology first put forward in \cite{BD99}).
 \begin{defn} Let $\mathcal{M}$ be a submanifold of $\R^n$. Then we say $\mathcal{M}$ is of \emph{Khintchine type} if for any decreasing  approximation function $\psi$ the set $\mathcal{S}_n(\psi)\cap\mathcal{M}$ satisfies the following dichotomy
  $$
| \mathcal{S}_n(\psi)\cap\mathcal{M}|_\mathcal{M}=\left\{\begin{array}{cll}
\text{N{\tiny ULL}}& \text{when}&\sum\psi(q)^n<\infty;\\
  \text{F{\tiny ULL}}& \text{when}&\sum \psi(q)^n=\infty.
 \end{array}
 \right.
 $$
 Here `F{\tiny ULL}' means the complement in $\mathcal{M}$ is a null set. 
 Also for convenience we say $\mathcal{M}$ is  of \emph{Khintchine type for convergence/divergence} if it satisfies the convergence/divergence statement above. 
 \end{defn}
 
 Clearly, a necessary condition for a manifold to be of Khintchine type for convergence is that it has to be extremal, since the latter simply corresponds to taking $\psi(q)=q^{-\tau}$ in the definition of Khintchine type manifolds. This means that the Khintchine type theory is deeper than the extremality theory in \S\ref{ExtremalManifolds}. 
 
 The two cases in the Khintchine type theory are usually proved by different approaches in the literature. For nondegenerate manifolds, the divergence case was settled first for planar curves by Beresnevich, Dickinson, and Velani \cite{BDV07} and later for manifolds in arbitrary dimensions by Beresnevich \cite{Ber12}; the convergence case was proved first for planar curves by Vaughan and Velani \cite{VV06} and for general nondegenerate manifolds  in a very recent breakthrough of Beresnevich and Yang \cite{BY23}. The interested readers are referred to the above papers and the references therein for more details about these spectacular works and various results predating them.

 \subsubsection{Khintchine type affine subspaces}
Next, let's turn our focus to the other stream. Unlike the situation for nondegenerate manifolds, the Khintchine type theory remained rather ad hoc for affine subspaces until recently. In 2000, Kovalevskaya showed that lines passing through the origin $\mathcal{L}_{\bm{\alpha},\bm{0}}$ (using the notation in \eqref{LineInRn}) is of Khintchine type for convergence if $\omega(\bm{\alpha})<n$ \cite{Kov00}.  Ram\'{i}rez studied coordinate hyperplanes (which are hyperplanes with one fixed coordinate) in 2015 \cite{Ram15}, and proved among other things that they are of Khintchine type for convergence if $\omega(\beta)<\frac{\sqrt{5}+1}2$, where $\beta$ is the value of the fixed coordinate. Note that in our notation, the parametrization matrix of a coordinate hyperplane is given by $A=(\beta,\bm{0})^\mathrm{T}\in\R^{n\times 1}$, so clearly $\omega(A)=\omega(\beta)$ in this case. Two years later, Ram\'{i}rez, Simmons, and S\"{u}ess  proved that coordinate lines $\mathcal{L}_{\bm{0},\bm{\beta}}$ are of Khintchine type\footnote{They actually only proved the divergence case, but their argument can be used to prove the convergence case as well in a straightforward manner.} if $\omega(\bm{\beta})<n$ \cite{RSS17}. In the same paper, they have also shown that all affine coordinate subspaces of dimension at least two are  of Khintchine type for divergence, absent any diophantine requirement. In 2021, Jason Liu\footnote{An 11th grader at the Davidson Academy when the paper was written.} and the author established the first general result addressing affine subspaces in arbitrary dimensions using a certain semi-multiplicative diophantine exponent \cite{HL21}. More precisely,  for $A\in\R^{d\times m}$
let $\omega^{\times}(A)$ be the supremum of all $\omega$ for which
$$
\left(\prod_{u = 1}^d\|\bm{q}\cdot\mathrm{row}_u(A)\|\right)^{1/d}<|\bm{q}|^{-\omega}
 $$
 has infinitely many solutions in $\bm{q}\in\Z^m$. It is proved in \cite{HL21} that  an affine subspace $\mathcal{L}$ of $\R^n$, parametrized by $A$ as in \eqref{AffineSubspace}, is of Khintchine type for convergence if $\min\{\omega^\times(A),\omega^\times(A')\}<n$. Soon after, by feeding some key counting arguments of \cite{HL21} into the ubiquity framework developed in \cite{BDV06}, Alvey proved the divergence counterpart under the stronger assumption  $\omega^{\scriptstyle\times}(A')<n$ \cite{Alv21}.

It does not take long to realize that $\omega^\times(A)$ is not the ``correct" exponent for this problem. For example, there are plenty of Khintchine type affine subspaces with $\omega^\times(A)=\infty$. This is because as long as one row of $A$ together with 1 is linearly dependent over $\Q$, it automatically forces  $\omega^\times(A)=\infty$ regardless what other rows are doing. In particular, all hyperplanes with at least one rational coefficient satisfy $\omega^\times(A)=\infty$, but we have already seen that some  coordinate hyperplanes are of Khintchine type.  The reason  $\omega^\times(A)$ in lieu of $\omega(A)$ is used in \cite{HL21} is purely due to some technical requirements by the method there. It transpires, from our next main result below, that $\omega(A)$ is still the right exponent to describe the Khintchine type theory for arbitrary affine subspaces, and moreover we  obtain the correct ``Khintchine threshold".

\begin{thm}\label{KT} Let $\mathcal{L}$ be an affine subspace parametrized by $A$ as in \eqref{AffineSubspace}, and suppose that $\omega(A)\neq n$. Then we have the following equivalence
$$
\mathcal{L}\text{ is of Khintchine type}\iff \omega(A)<n.
$$
\end{thm}

All the existing results to date in this regard fall into the realm of Theorem \ref{KT}. One may easily check that the diophantine conditions in \cite{Kov00,RSS17} matches that of Theorem \ref{KT} when specializing $\mathcal{L}$ to lines passing through the origin $\mathcal{L}_{\bm{\alpha},\bm{0}}$ and coordinate lines $\mathcal{L}_{\bm{0},\bm{\beta}}$ respectively, while the condition on coordinate hyperplanes in \cite{Ram15} is significantly improved from $\omega(\beta)<\frac{\sqrt{5}+1}2$ to $\omega(\beta)<n$ here. Moreover, the main results of \cite{HL21,Alv21} are covered  by Theorem \ref{KT} as well since it follows by definition that $\omega(A)\le \omega^\times(A)$ and $\omega(A)\le \omega(A')\le\omega^\times(A')$.

Theorem \ref{KT} gives a complete classification of all Khintchine type affine subspaces of $\R^n$ except for the boundary case $\omega(A)=n$. As a matter of fact, we will see that Theorem \ref{KT} has squeezed everything out of the exponent $\omega(A)$ as far as the Khintchine type theory is concerned.  

Before we go on to investigate the exceptional case $\omega(A)=n$, some comments are in order. An alert reader may have already realized that some diophantine condition is only absolutely necessary for affine subspaces to be of Khintchine type for convergence. It follows from \eqref{SigmaLowerBound}  that if $\omega(A)>n$ then  $\mathcal{L}\subset\mathcal{S}_n\left(1/n+\varepsilon\right)$ for some $\varepsilon>0$, that is to say, all points of $\mathcal{L}$ are simultaneously very well approximable and hence $\mathcal{L}$ fails to be of Khintchine type for convergence in a dramatic fashion. 
 On the other hand, there does not appear to be any real bottleneck to prevent affine subspaces from being Khintchine type for divergence, as evidenced by the aforementioned result in \cite{RSS17} on  coordinate hyperplanes of dimension at least two. In fact, as remarked in \cite{RSS17}, the inclusion $\mathcal{L}\subset\mathcal{S}_n\left(1/n+\varepsilon\right)$ is philosophically ``almost as good" as being of Khintchine type for divergence, since the former implies the latter for some ``nice" approximation functions $\psi$ (roughly speaking it  means $\psi$ decays in an orderly manner). We refer the readers to the remark below \cite[Theorem 2]{RSS17} for the precise definition and some explicit examples. 
 For the reasons outlined, it is generally believed that all affine subspaces are of Khintchine type for divergence. 
As a result, we may focus our quest for the optimal Khintchine threshold completely on the convergence side.

As has been said, the exponent $\omega(A)$ can only take us as far as Theorem \ref{KT}, since it will be soon revealed that there exist Khintchine type affine subspaces as well as  non-Khintchine type ones when $\omega(A)=n$. 
To state our next result in precise terms, we need the help of a secondary exponent $\omega'(A)$ which sees through the fine structure of $\omega(A)$ at logarithmic scales. 

\begin{defn} For a  real matrix $A\in\R^{d\times m}$ which is diophantine (i.e. $\omega(A)<\infty$), 
the \emph{logarithmic diophantine exponent} $\omega'(A)$ is defined as the supremum of the set of $\omega'\in\R$ such that, for infinitely many $\bm{q}\in\Z^m$, we have
$$
\|A\bm{q}^\mathrm{T}\|<|\bm{q}|^{-\omega(A)}(\log |\bm{q}|)^{-\omega'} .
$$
\end{defn}

Note that
here and throughout the paper, we redefine the natural logarithmic function so that $\log t=1$ when $t\le e$.
Also, it has been shown in \cite{BDV01} that with a given diophantine exponent there exist matrices with any prescribed  number as their logarithmic exponent, so the definition above is not vacuous.

 \begin{thm}\label{CriticalKTC}
Let $\mathcal{L}$ be an affine subspace parametrized by $A$ as in  \eqref{AffineSubspace} with  $\omega(A)=n$. Then
\begin{enumerate}[(a)]
  \item If $\omega'(A)<-1$, then $\mathcal{L}$ is  of Khintchine type for convergence;
    \item If $\omega'(A)>n$, then $\mathcal{L}$ is not of Khintchine type for convergence.
\end{enumerate}
\end{thm}

This theorem narrows down the exact Khintchine threshold within two logarithmic scales at $\omega(A)=n$, though it is not at all clear which end is closer to the truth. 

\begin{rem}
When $\omega'(A)=-1$ or $n$, our work actually allows the possibility of bringing iterated logarithms into play, but we decide to settle for the present form for simplicity. 
\end{rem}

\begin{rem}
 We may combine  Theorem \ref{CriticalKTC}(b) with Theorem \ref{ExtremalCondition} to confirm the existence of extremal manifolds that are not of Khintchine type for convergence (hence not of Khintchine type), something not previously known. This is a new phenomenon only pertaining to affine subspaces since nondegenerate manifolds are known to be both extremal and of Khintchine type.
\end{rem}

\begin{rem}\label{Groshev}
We remark in passing that there is a parallel theory  for dual approximation, which is usually referred to as the Groshev type theory. 
 Ghosh  has proved that an affine subspace $\mathcal{L}$ is of Groshev type for convergence if $\omega_j(A)<n$ for all $1\le j\le m$ \cite{Gho11}, where these $\omega_j(A)$ are the same higher order exponents introduced in \cite{Kle03,Kle08}. The divergence counterpart has recently been proved under the same condition in the more general inhomogeneous setting by Beresnevich, Ganguly, Ghosh, and Velani \cite{BGGV20}. Comparing the above diophantine condition with that of Theorem \ref{KT}, it is very natural to ask whether the presence of higher order exponents  is again redundant as is the case for the extremality theory. Namely is it true that the sole condition $\omega(A)<n$ should be sufficient to deliver that $\mathcal{L}$ is of Groshev type? If the answer is yes, then we would bring the development of the Groshev type theory in line with that of the Khintchine type theory; if not, this would lead to the existence of manifolds of Khintchine type but not of Groshev type, something completely unheard of! Either way, we would have very interesting findings, but we believe that the former is the much more likely scenario. 
\end{rem}

\subsubsection{Strong Khintchine type manifolds}

Sometimes, the convergence aspect of a Khintchine type manifold can be proved without assuming that $\psi$ is monotonic, in which case we say such a manifold is of \emph{strong Khintchine type for convergence}. Even though the original Khintchine's theorem does not need to assume $\psi$ is monotonic for the convergence statement to hold,  when Khintchine type theorems are proved for manifolds, the monotonicity  plays a key role in one part of the argument by allowing one to average over $q$.  It is exactly this averaging step that renders it possible to employ a certain rational points counting estimate with varying denominators in lieu of a fixed denominator, the latter of which sometimes is insufficient to deliver the desired result. For this reason, it makes sense to introduce the above definition to test how far the Khintchine type theory can go without the monotonicity assumption on the approximation function.

   Indeed, nondegenerate manifolds  satisfying various curvature/rank conditions have been shown to be of strong Khintchine type for convergence \cite{DRV91,BVVZ17,  Sim18}. But these conditions completely exclude curves.  The only nondegenerate curve that  is known to be of strong Khintchine type for convergence to date  is the parabola \cite{Hua20a}. On the contrary to this situation for nondegenerate manifolds, we prove in the result below that the vast majority of affine subspaces are actually of strong Khintchine type for convergence. 

\begin{thm}\label{StrongKTC} Let $\mathcal{L}$ be an affine subspace parametrized by $A$ as in  \eqref{AffineSubspace}. Then
$\mathcal{L}$ is of strong Khintchine type for convergence if $\omega(A') < n$.
\end{thm}
\begin{rem}\label{StrongKTC_VS_KTC}
Unlike the situation for the Khintchine type theory where the exact threshold in terms of $\omega(A)$ has been discovered, it is not clear what happens when $\omega(A')>n$ in this new context.  It might be tempting to speculate that in such cases  $\mathcal{L}$ is not of strong Khintchine type for convergence.  If that is indeed the case, we may find $\mathcal{L}$ with $\omega(A')>n>\omega(A)$ and then claim that it is of Khintchine type for convergence but not of strong Khintchine type for convergence. This would be very interesting to know, since currently we do not have evidence supporting or against the existence of such manifolds. As a simpler toy model, one may look at affine coordinate subspaces ($A'$ is zero matrix) first, and hopes to gain some insights into this problem. 
\end{rem}

\subsubsection{Multiplicative approximation on planar lines}
Consider the planar line
$$
\mathcal{L}_{\alpha,\beta}:=\left\{(x,y)\in\mathbb{R}^2\bigm| y=\alpha x+\beta\right\}.
$$
Recently, Chow and Yang \cite[Theorem 1.2]{CY} proved the analogue of Gallagher's theorem on multiplicative approximation for $\mL_{\alpha,\beta}$, that is, almost every point $(x,y)$ on $\mL_{\alpha,\beta}$ satisfies
\begin{equation}\label{GallagherLine}
\liminf_{q\to\infty}q(\log q)^2\|qx\|\|qy\|=0.
\end{equation}

In general, we may consider the set of multiplicatively $\psi$-approximable points on the plane
\begin{equation}\label{MultiPsiAppr}
\mathcal{S}^\times_2(\psi):=\left\{(x,y)\in\R^2: \|qx\|\|qy\|<\psi(q)\textrm{ for infinitely many }{q}\in\N\right\}
\end{equation}
and we say a planar curve $\mathcal{C}$ is of \emph{multiplicative Khintchine type for convergence/divergence} if 
$$
|\mathcal{S}^\times_2(\psi)\cap\mathcal{C}|_{\mathcal{C}}=\text{N\tiny{ULL}}/\text{F\tiny{ULL}}
$$
for any decreasing function $\psi: \mathbb{N}\rightarrow \mathbb{R}_{>0}$
  such that the series
\begin{equation}\label{MultiSeries}
\sum_{q=1}^\infty
 \psi(q) \log q
  \end{equation}
   is convergent/divergent.
 It is, for instance, shown in \cite{BL07} that a nondegenerate planar curve is of multiplicative Khintchine type for convergence. 
 One can see that \eqref{GallagherLine} just falls short of establishing that $\mathcal{L}_{\alpha,\beta}$ is of multiplicative Khintchine type for divergence. To complement \eqref{GallagherLine}, Chow and Yang have also proved that if
\begin{equation}\label{ChowYang}
\omega(\alpha,\beta)<5\;\text{ or }\;\omega^\times{\alpha\choose\beta}<2
\end{equation}
then $\mathcal{L}_{\alpha,\beta}$ is of multiplicative Khintchine type for convergence \cite[Theorems 1.3 and 1.4]{CY}. 
An immediate consequence of this is that the  log-squared factor in \eqref{GallagherLine} is best possible in the sense  that $(\log q)^2$ cannot be replaced with $(\log q)^{2+\varepsilon}$ for any $\varepsilon>0$, under either conditions in \eqref{ChowYang}. In our next theorem, an essentially optimal condition is found for $\mathcal{L}_{\alpha,\beta}$ to be of multiplicative Khintchine type for convergence. Actually, our result is proved in  the more general inhomogeneous setting by introducing a shift $(\theta_1,\theta_2)$ in the definition \eqref{MultiPsiAppr} of $\mathcal{S}^\times_2(\psi)$.

\begin{thm}\label{MultiApprLine}
Let $(\theta_1,\theta_2)\in\R^2$ and suppose that
$(\alpha,\beta)\in\mathbb{R}^2$ satisfies 
\begin{equation}\label{Huang}
\omega{\alpha\choose\beta}<2\;\text{ and }\;\alpha\neq0.
\end{equation}
Then for  almost all $(x,y)\in\mathcal{L}_{\alpha,\beta}$ there are  only finitely many $q\in\mathbb{N}$ such that
$$
\|qx-\theta_1\|\|qy-\theta_2\|<\psi(q)
$$
if the series \eqref{MultiSeries} converges. In particular, 
$\mathcal{L}_{\alpha,\beta}$ 
is of multiplicative Khintchine type for convergence. 
\end{thm}

We first of all point out that  \eqref{Huang} is indeed a more general assumption  than \eqref{ChowYang}, as the former is implied by the latter. Clearly both conditions in \eqref{ChowYang}  requires $\alpha\neq0$. Also, we see that by Khintchine's transference principle \eqref{TransferPrinciple} $\omega(\alpha,\beta)<5\implies\omega{\alpha\choose\beta}<2$, and that by definition   $\omega{\alpha\choose\beta}\le \omega^\times{\alpha\choose\beta}$. So \eqref{Huang} does relax \eqref{ChowYang}. In addition, to have a more intuitive comparison on how generic the various diophantine conditions in \eqref{ChowYang} and \eqref{Huang} are, we record the following dimensional statements on their respective ``exceptional sets" (see \cite{BD86,BV15}): 
$$
\dim\left\{(\alpha,\beta)\in\R^2\bigm|\omega(\alpha,\beta)\ge5\right\}=\frac32;
$$
$$
\dim\left\{(\alpha,\beta)\in\R^2\Bigm|\omega^\times{\alpha\choose\beta}\ge2\right\}=\frac75;
$$
$$
\dim\left\{(\alpha,\beta)\in\R^2\Bigm|\omega{\alpha\choose\beta}\ge2\right\}=1.
$$

Next, we observe that Theorem \ref{MultiApprLine} is best possible except for the boundary case $\omega{\alpha\choose\beta}=2$. By a theorem of Kleinbock \cite[Corollary 5.7]{Kle03}, we know that if $\omega{\alpha\choose\beta}>2$ then $\mathcal{L}_{\alpha,\beta}$ is not strongly extremal, that is, there exists an $\varepsilon>0$ such that
$$
|\mathcal{S}^\times_2(1+\varepsilon)\cap\mathcal{L}_{\alpha,\beta}|_{\mathcal{L}_{\alpha,\beta}}>0
$$
where as usual we write  $\mathcal{S}^\times_2(\tau)$ for $\mathcal{S}^\times_2(q\to q^{-\tau})$. Hence $\mathcal{L}_{\alpha,\beta}$ is not of multiplicative Khintchine type for convergence, which explains the almost sharpness of the first condition in \eqref{Huang}. Also the second condition $\alpha\neq0$ cannot be dispensed with from \eqref{Huang}. Otherwise when $\alpha=0$, \eqref{Huang} becomes    $\omega{0\choose\beta}=\omega(\beta)<2$. But
 we see that $\mathcal{L}_{0,\beta}\subset\mathcal{S}^\times_2(\tau)$ for all $\tau<\omega(\beta)$, which provides  counterexamples (that are not strongly extremal) as long as $\omega(\beta)>1$. This shows we have to assume $\alpha\neq0$ unless $\omega(\beta)=1$. But it turns out that even the much stronger condition $\omega(\beta)=1$ cannot get the job done by itself. In fact, Beresnevich, Haynes, and Velani have to assume $\beta$ is badly approximable (meaning that $\liminf q\|q\beta\|>0$ which is even more restrictive than $\omega(\beta)=1$ and no longer holds for typical numbers)  to prove that $\mL_{0,\beta}$ is of multiplicative Khintchine type for convergence \cite[Corollary 2.1]{BHV20}. 
 
 Once equipped with a key counting result (Theorem \ref{UpperBound}) below, the proof of Theorem \ref{MultiApprLine} follows by the methodology in \cite{BL07}.  See  \cite[Theorem 2]{BV15} for a related inhomogeneous theorem in the context of nondegenerate planar curves.

\subsection{Jarn\'{i}k type manifolds}

In addition to the Khintchine type theory, one may as well seek Hausdorff measure theoretic statements for the  set $\mathcal{M}\cap\mathcal{S}_n(\psi)$ of $\psi$-approximable points on a manifold $\mathcal{M}$. This turns out to be a much more daunting task, as it requires deeper understanding of the distribution of rational points lying close to $\mathcal{M}$. A general phenomenon is that the faster  $\psi$ is allowed to decay the more difficult will it be to prove universal results for large class of manifolds. This is partly due to our poor knowledge on the distribution of rational points near manifolds at very fine scales. When $\psi$ is not too small, the problem is expected  to be governed by geometric/analytic properties of the manifold; however once $\psi$ drops below a certain threshold, the arithmetic of the manifold starts to kick in, which further complicates the problem.  With these being said, 
 there has been, nevertheless,  some exciting developments for nondegenerate submanifolds \cite{BDV07,VV06,Ber12,Hua20b}. In the following we study this problem for affine subspaces.

Our first result in this section is the convergence part of an inhomogeneous Jarn\'{i}k type theorem. Recall that the set of inhomogeneously $\psi$-approximable points  $\mathcal{S}_n(\psi,\bm{\theta})$ is defined in \eqref{InhomoPsiAppr}.

\begin{thm}\label{JTC} 
 Let $\mathcal{L}$ be an affine subspace parametrized by $A$ as in  \eqref{AffineSubspace}. Suppose that $A$ has diophantine exponent  $\omega(A)=\omega<n$, and let
 \begin{equation}\label{RangeOfS-Conv}
  \frac{\omega(d+1)-m}{\omega+1}<s\le d.
 \end{equation}
Then for any decreasing function $\psi:\N\to\R_{>0}$ such that
\begin{equation}\label{ConvCondition}
    \sum_{q=1}^\infty\psi(q)^{m+ s}q^{d - s}<\infty,
\end{equation}
and for any $\bm{\theta}\in\R^n$, we have
$$\mathcal{H}^s(\mathcal{S}_n(\psi,\bm{\theta}) \cap \mathcal{L}) = 0.
$$
\end{thm}

   \begin{rem} \label{JTC-is-optimal?}
   For a given $s\le d$, \eqref{RangeOfS-Conv} can be rewritten as 
   $
   \omega(A)<\frac{m+s}{d+1-s}
   $, viewed as a condition 
   for $\omega(A)$. A similar application of Theorem \ref{JTC-Phi} as in the proof of Theorem \ref{CriticalKTC}(a) reveals that the conclusion of Theorem \ref{JTC} still holds in the limit case   $\omega(A)=\frac{m+s}{d+1-s}$ (or equivalently $s=\frac{\omega(d+1)-m}{\omega+1}$) as long as $\omega'(A)<-\frac1{d+1-s}$. However, it is
   in general not clear what happens if $\omega(A)>\frac{m+s}{d+1-s}$, beyond the  case $s=d$ on the Khintchine type theory that has been treated earlier in Theorem \ref{KT}.
 \end{rem}

In the homogeneous case, we also have a divergence counterpart.  
\begin{thm}\label{JTD}
Let $\mathcal{L}$ be an affine subspace parametrized by $A$ as in  \eqref{AffineSubspace}. Suppose that $A$ has diophantine exponent  $\omega(A)=\omega<n$,  and let
 \begin{equation}\label{RangeOfS-Div}
  \frac{d+1-m\tau_0}{\tau_0+1}<s\le d,
 \end{equation}
 where
  \begin{equation}\label{tau0}
 \tau_0:=\frac1{\max\{m,\omega\}}.
 \end{equation}
Then for any decreasing function $\psi:\N\to\R_{>0}$ such that
\begin{equation}\label{DivCondition}
    \sum_{q=1}^\infty\psi(q)^{m+ s}q^{d - s}=\infty,
\end{equation}
we have
$$\mathcal{H}^s(\mathcal{S}_n(\psi) \cap \mathcal{L}) = \infty.
$$
For the case $s=d$, the conclusion holds in the stronger sense that $\mathcal{S}_n(\psi) \cap \mathcal{L}$ has full $d$-dimensional Hausdorff measure in $\mathcal{L}$.
\end{thm}

It remains to be seen whether our ubiquity approach can be modified to prove the inhomogeneous version of Theorem \ref{JTD} (that is, replacing $\mathcal{S}_n(\psi)$ by $\mathcal{S}_n(\psi,\bm{\theta})$ ). Previously this has been successfully done for nondegenerate curves using  techniques from geometry of numbers \cite{BVV11,BVVZ21}.

The range \eqref{RangeOfS-Div} is the same as \eqref{RangeOfS-Conv} when $\omega\ge m$, but the former is shorter when $\omega<m$. So  the above two theorems, when combined together,  give a complete Jarn\'{i}k type theorem for affine subspaces with $s$ in the range \eqref{RangeOfS-Div}.

\begin{cor}\label{JT}
 Let $\mathcal{L}$ and $A$ be the same as with Theorem \ref{JTC}, and suppose that $s$ is a real number  satisfying \eqref{RangeOfS-Div}. Then  for any decreasing function $\psi:\N\to\R_{>0}$ we have 
 \[\displaystyle
 \mathcal{H}^s(\mathcal{S}_n(\psi) \cap \mathcal{L}) = 
 \begin{cases}
 0&\text{ when } \sum\psi(q)^{m+ s}q^{d - s}<\infty;\\[4pt]
 \infty&\text{ when } \sum\psi(q)^{m+ s}q^{d - s}=\infty.
 \end{cases}
 \]
\end{cor}

Note that  the reverse implication part of Theorem
\ref{KT} (the main substance therein), namely, $\mathcal{L}$ is of Khintchine type  when $\omega(A)<n$, just corresponds to the special case $s=d$ of Corollary \ref{JT}, since the $d$-dimensional Hausdorff measure $\mathcal{H}^d$ is comparable to the induced Lebesgue measure $|\cdot|_\mathcal{L}$ on $\mathcal{L}$.

Next we investigate the possibility of removing the monotonicity assumption on the approximation function $\psi$. As a result, this would permit us to deal with the general setup that $\psi$ is supported in any given infinite subset $\mathcal{Q}$ of $\N$. To this end, define
$$
\nu(\mathcal{Q}):=\inf\bigg\{\nu\in\R\biggm|\sum_{q\in\mathcal{Q}}q^{-\nu}<\infty\bigg\}.
$$
It is easy to see that the quantity $\nu(\mathcal{Q})\in[0,1]$ measures the sparseness of $\mathcal{Q}$.

\begin{thm}\label{StrongJTC}
Let $\mathcal{L}$ be an affine subspace parametrized by $A$ as in  \eqref{AffineSubspace},  $\mathcal{Q}$ be an infinite subset of $\N$ with $\nu(\mathcal{Q})=\nu$.  Suppose that $A'$ has diophantine exponent  $\omega(A')=\omega<\frac{n}\nu$,  and let
 \begin{equation}\label{RangeOfS-Strong}
  \frac{\omega(d+\nu)-m}{\omega+1}<s\le d.
 \end{equation}
Then for any (not necessarily decreasing) function $\psi:\mathcal{Q}\to\R_{\ge0}$ such that
\begin{equation}\label{ConvCondition-Strong}
    \sum_{q\in\mathcal{Q}}\psi(q)^{m+ s}q^{d - s}<\infty,
\end{equation}
and for any $\bm{\theta}\in\R^n$, we have
$$\mathcal{H}^s(\mathcal{S}_n(\psi,\bm{\theta}) \cap \mathcal{L}) = 0.
$$
\end{thm}
Clearly our strong Khintchine type theorem (Theorem \ref{StrongKTC}) is the special case of Theorem \ref{StrongJTC} with $\mathcal{Q}=\N$, $s=d$, and $\bm{\theta}=\bm{0}$.

It is interesting to compare the respective ranges  \eqref{RangeOfS-Conv} and \eqref{RangeOfS-Strong} of $s$. Recall that the typical values of $\omega(A)$ and $\omega(A')$ are $m/(d+1)$ and $m/d$ respectively. Then note that for typical $A$, \eqref{RangeOfS-Conv} simply becomes $0<s\le d$ which covers the full range. On the other hand, if $\mathcal{Q}=\N$, then $\nu(\mathcal{Q})=1$ and \eqref{RangeOfS-Strong} becomes $m/n<s\le d$. If  $\mathcal{Q}$ is
 a super-polynomial sequence (meaning that it grows faster than any polynomial), such as $\{n^{\lfloor\log{n}\rfloor}\}$ or any lacunary sequence, it is easy to see that $\nu(\mathcal{Q})=0$. Therefore for such a  sequence $\mathcal{Q}$ and for a typical $A'$, \eqref{RangeOfS-Strong} becomes $0<s\le d$, which again covers the full range. 

As another interesting consequence, Theorem \ref{StrongJTC} applies when  $\nu(\mathcal{Q})=0$, $\omega(A')<\infty$, and $s=d$ . More precisely, under the rather mild condition that $A'$ is diophantine,  for any infinite subset $\mathcal{Q}$ of $\N$ with $\nu(\mathcal{Q})=0$ and for any function $\psi:\mathcal{Q}\to\R_{\ge0}$ such that $\sum_{q\in\mathcal{Q}}\psi(q)^n<\infty$, we have
$$
|\mathcal{S}_n(\psi,\bm{\theta}) \cap \mathcal{L}|_\mL = 0.
$$
In other words, this shows that when $\omega(A')<\infty$,  $\mathcal{L}$ satisfies the convergence part of the inhomogeneous Khintchine type theorem provided that the approximation function $\psi$, which needs not be decreasing, is supported on a super-polynomial sequence $\mathcal{Q}$.

\subsection{Hausdorff dimension of very well approximable vectors}
Applying our Jarn\'{i}k type theorems (Theorems \ref{JTC} and \ref{JTD}) and our main counting results in the next section (Theorems \ref{UpperBound} and \ref{DualUpperBound}),   we  can obtain  the following result for the Hausdorff dimension of the set $\mathcal{S}_n(\tau,\bm{\theta})\cap\mathcal{L}$.

\begin{thm}\label{Dimension}
Let $\mathcal{L}$ be an affine subspace parametrized by $A$ as in \eqref{AffineSubspace}, and suppose that $\omega(A)=\omega$ and $\sigma(A)=\sigma$. Then for any $\bm{\theta}\in\R^n$
we have
$$\everymath={\displaystyle}
\dim\mathcal{S}_n(\tau,\bm{\theta})\cap\mathcal{L}\le\left\{
\begin{aligned}
\frac{n+1}{\tau+1}-m\;\;&\quad {when }\quad\frac1n\le\tau\le\frac1{\omega};\\
\frac{(d+1)\omega-m}{(\tau+1)\omega}& \quad {when }\quad\frac1{\omega}\le\tau\le\frac1{\omega}+\sigma-\frac{d+1}m;\\
\frac{m(\sigma-\tau)}{\tau+1}\quad&\quad\text{when}\quad\frac1{\omega}+\sigma-\frac{d+1}m\le\tau\le\sigma;\\
0\quad\qquad&\quad\text{when}\quad\sigma\le\tau.
\end{aligned}
\right.
$$
Moreover, in the homogeneous case $\bm{\theta}=0$ we have the exact formula
\begin{equation}\label{DimFormula}
\dim\mathcal{S}_n(\tau)\cap\mathcal{L}=\frac{n+1}{\tau+1}-m\quad {when }\;\;\frac1n\le\tau<\tau_0,
\end{equation}
where $\tau_0$ is given in \eqref{tau0}.
\end{thm}
As usual, the equation regarding $\dim\mathcal{S}_n(\tau)\cap\mathcal{L}$ follows from sharp upper bound and  lower bound estimates coming from Theorems \ref{JTC} and \ref{JTD} respectively in a straightforward manner. So it only remains to prove the first part of Theorem \ref{Dimension}, which we will take up in \S \ref{ProofDimension}.

\begin{rem}
When $\tau>\sigma$, Theorem \ref{DualUpperBound} below actually allows us to show the stronger conclusion that $\mathcal{S}_n(\tau,\bm{\theta})\cap\mathcal{L}$ is a finite set, hence its Hausdorff zero-measure is bounded. 
\end{rem}

\begin{rem}\label{ExactHausdorffDim}
A lower bound of the form
\begin{equation}\label{DimLowerBound}
\dim\mathcal{S}_n(\tau)\cap\mathcal{M}\ge\frac{n+1}{\tau+1}-m\quad {when }\;\;\frac1n\le\tau<\frac1m,
\end{equation}
has been obtained by Beresnevich, Lee, Vaughan, and Velani for arbitrary $C^2$ manifold $\mathcal{M}$ \cite{BLVV17} based on the mass transference principle of Beresnevich and Velani \cite{BV06}.  When $\omega>m$, there is a discrepancy between the ranges of $\tau$ in \eqref{DimFormula} and \eqref{DimLowerBound}. It is reasonable to suspect that \eqref{DimFormula} holds for $1/n\le \tau\le1/\omega$, that is, the first upper bound in Theorem \ref{Dimension} is actually an equality. Moreover we also believe this is the largest range of $\tau$ for which \eqref{DimFormula} holds. It is not clear whether the second and third upper bounds in Theorem \ref{Dimension} are  sharp, and if not what  the exact formula for $\dim\mathcal{S}_n(\tau)\cap\mathcal{L}$ looks like when $1/\omega<\tau<\sigma$.
\end{rem}

\section{Counting rational points near affine subspaces}
\subsection{Main counting results}\label{CountingResults}
It is well known that obtaining sharp bounds for the number of rational points near a manifold lies at the core of establishing Khintchine and Jarn\'{i}k type theorems for the manifold.

Let $\mathcal{M}$ be a $d$-dimensional submanifold  of $\mathbb{R}^n$,  parametrized in the Monge form $$\big(\bm{x},\bm{f}(\bm{x})\big)\in\R^d\times\R^m,\quad \bm{x}\in\mathcal{U}$$ 
with 
$$
\mathcal{U}:=[0,1]^d.
$$
 Then for  $Q\ge1$ and $0<\delta\le1/2$, consider the  counting function
\begin{equation}\label{CountingFunctionGen}
N_{\mathcal{M}}(Q,\delta):=\#\Big\{\bm{a}/q\in\mathcal{U}\cap\mathbb{Q}^d\bigm|1\le q\le Q, \,\|q\bm{f}(\bm{a}/q)\|<\delta\Big\},
\end{equation}
which counts the number of rational points of denominator $q\le Q$  that lies within distance $O(\delta/q)$ to $\mathcal{M}$. 
The central problem here is to give  good estimates  for this counting function. Of course, if $\mathcal{M}$ is a rational affine subspace, then $N_{\mathcal{M}}(Q,\delta)\asymp Q^{d+1}$ which is sort of the worst case scenario. In practice, one seeks estimates better than the trivial bound $O(Q^{d+1})$, assuming various geometric, analytic, and arithmetic conditions on $\mathcal{M}$. 
One commonly used condition is the notion of nondegeneracy. We say $\mathcal{M}$ is \emph{$l$-nondegenerate} at $\bm{x}$ if the partial derivatives of the map $\bm{x}\to\big(\bm{x},\bm{f}(\bm{x})\big)$ at $\bm{x}$ of order up to $l$ generate $\R^n$. In geometrical terms, this means exactly that the order of contact that $\mathcal{M}$ has  with any hyperplane is at most $l$. Roughly speaking, nondegeneracy  requires that $\mathcal{M}$ not being too flat locally.

A simple probabilistic heuristic shows that one would expect  
$$N_{\mathcal{M}}(Q,\delta)\asymp \delta^mQ^{d+1},$$
if rational points somehow distribute in a ``random" fashion near $\mathcal{M}$. As before $m=n-d$ denotes the codimension of $\mathcal{M}$ in $\R^n$. It is not hard to see from concrete examples that we cannot expect this randomness to dominate all the way, and that it has to give way to the structure caused by the underlying manifold when $\delta$ drops below a certain fine scale.  Unfortunately, we do not currently have a full picture for the expected behavior of the counting function, as this phenomenon is quite manifold specific and hard to capture in simple terms. However, when $\delta$ is not too small, we put forward the following plausible conjecture.
\begin{conj}\label{MainConjecture}
Suppose that a $d$-dimensional submanifold $\mathcal{M}$ of $\R^n$ is $l$-nondegenerate everywhere with $l=m+1$. Then
$$
N_\mathcal{M}(Q,\delta)\ll_\mathcal{M} \delta^mQ^{d+1}\quad \text{when }\delta\ge Q^{-\frac1m+\varepsilon}\text{ and }Q\to\infty.
$$
\end{conj}

A matching lower bound has been established by Beresnevich for nondegenerate analytic manifolds \cite{Ber12}, and the analyticity condition has recently been removed for curves \cite{BVVZ21}. So we may even make a more ambitious conjecture that  the above upper bound can be turned into an asymptotic formula 
\begin{equation}\label{AsympConjecture}
N_\mathcal{M}(Q,\delta)\sim c_\mathcal{M} \delta^mQ^{d+1}\quad\text{when }\delta\ge Q^{-\frac1m+\varepsilon}\text{ and }Q\to\infty.
\end{equation}

Now we discuss the current status of Conjecture \ref{MainConjecture}. For planar curves, it has been proved  by Vaughan and Velani \cite{VV06} and in the stronger asymptotic form \eqref{AsympConjecture} by the author \cite{Hua15}. To the best of our knowledge, it is not even known for any single curve of codimension at least two, except for some progress on space curves by the author \cite{Hua19}. We have a bit more success for manifolds of higher dimension. The conjecture has been proved by the author for hypersurfaces with nonvanishing Gaussian curvature \cite{Hua20b}. Note that this result does not exhaust all hypersurfaces addressed in Conjecture \ref{MainConjecture}. The conjectural 2-nondegeneracy condition for hypersurfaces is equivalent to saying that the rank of the Hessian matrix is at least one, while the nonvanishing curvature condition requires that the Hessian matrix has full rank. So even in the case of hypersurfaces, there remains some work to be done in order to completely resolve the conjecture.  The method in \cite{Hua20b} has been subsequently extended to cover  some  surfaces of codimension higher than 1 by Schindler and Yamagishi \cite{SY22}, and by Munkelt \cite{Mun}. It is in particular shown in these works that for manifolds satisfying certain rank conditions on their Hessian matrices, Conjecture \ref{MainConjecture} may hold for $\delta$ beyond the conjectured range. 

We also remark in passing that there is some close connection between the counting problem in question and Serre's dimension growth conjecture, originally stated in the context of projective varieties. Actually one may view Conjecture \ref{MainConjecture} as a perturbed version of the dimension growth conjecture for smooth manifolds, which includes the original dimension growth conjecture as a special case. We do not plan to go into details in this regard, and only refer the interested readers to the above cited papers \cite{Hua20b,SY22,Mun} for the precise statements.

Thus, the primary objective of this section is to address the analogue of Conjecture \ref{MainConjecture} for affine subspaces. We will demonstrate that our counting theorems, which we will unveil below, lead to all the novel results about diophantine approximation on affine subspaces in \S \ref{DioExpExtremal} and \S\ref{KhintchineJarnik}.

For any matrix $A\in\R^{(d+1)\times m}$ and for an inhomogeneous shift
$$
\bm{\theta}=(\theta_1,\ldots,\theta_d,\theta_{d+1},\ldots,\theta_{d+m})=(\bm{\theta}^{(1)}, \bm{\theta}^{(2)})\in\R^{d}\times\R^m,
$$
define
$$N_A(Q, \delta,\bm{\theta}) := \#\left\{(q,\bm{a})\in\Z^{d+1}:  \left\|(q, \bm{a}+\bm{\theta}^{(1)})A-\bm{\theta}^{(2)}\right\| < \delta, \,|(q,\bm{a})|< Q \right\}$$
and
$$N_{A}'(Q, \delta,\bm{\theta}) :=\sup_{q\in\R} \#\left\{\bm{a}\in\Z^d:  \left\|(q, \bm{a}+\bm{\theta}^{(1)})A-\bm{\theta}^{(2)}\right\| < \delta,\, |\bm{a}|< Q \right\}.$$
It is easy to see that the first counting functions is the inhomogeneous analogue of \eqref{CountingFunctionGen} for affine subspaces, while the second function is a variation of the first which only counts rational points with a fixed denominator $q$.

As mentioned above, flatness in general works against the prospect of proving desirable estimates.  It can however be turned to work to our advantage if certain diophantine conditions
are on our side. For instance, this phenomenon has already appeared in the work of Hardy and Littlewood \cite{HL22} dating back to the 1920s, where they have extensively studied the lattice point problem for  right triangles. The lattice rest (the difference between the actual count and the expected count of lattice points) in this problem is intimately related to the bad approximability of the slope of the hypotenuse.  Back to our problem, it is therefore expected that the counting functions above should depend on the diophantine properties of $A$ in a crucial way.

 Before  stating   the main counting results in this paper, we need one more definition. For a decreasing  function $\phi$: $\R_{>0}\to\R_{>0}$, we say $A$ is $\phi$-badly approximable if for all nonzero $\bm{q}\in\Z^m$ we have
$$
\left\|A\bm{q}^\mathrm{T}\right\|\ge\phi(|\bm{q}|).
$$
Apparently, the notion of $\phi$-bad approximability generalizes the diophantine exponent defined earlier, and it can control the diophantine property of $A$ in a more precise manner. 

\begin{thm}\label{UpperBound-Phi}   Let $\phi$: $\R_{>0}\to\R_{>0}$ be a decreasing  function, $\bm{\theta}\in\R^{d+m}$, and denote 
\begin{equation}\label{UpperConstant}
C_{d,m}:=8^{d}\pi^{2m}.
\end{equation}
\begin{enumerate}[(a)]
\item
Suppose that  $A\in \R^{(d+1)\times m}$ is  $\phi$-badly approximable.
Then for any positive integer $Q$ and  positive real number $\delta$ we have 
$$N_A(Q, \delta,\bm{\theta})\le C_{d+1,m} \,\delta^m\max\left\{Q,\,\frac1{\phi\left(\delta^{-1}\right)}\right\}^{d+1}.$$

\item
If instead, we assume the stronger condition that $A'$ is $\phi$-badly approximable,  then 
$$
N_{A}'(Q, \delta,\bm{\theta})\le C_{d,m}\, \delta^m\max\left\{Q,\,\frac1{\phi\left(\delta^{-1}\right)}\right\}^{d}.
$$
\end{enumerate}
\end{thm}

The following theorem is a special case of Theorem \ref{UpperBound-Phi} in terms of the diophantine exponent of $A$.
 
 \begin{thm}\label{UpperBound}
Let  $A$ be a $(d+1)\times m$ real matrix, $\bm{\theta}\in\R^{d+m}$, $\varepsilon$ be a positive number, and $C_{d,m}$ be the constant defined in \eqref{UpperConstant}.
 \begin{enumerate}[(a)]
     \item 
     Suppose that $\omega(A)=\omega$. Then  for $\delta\ge Q^{-\frac1\omega+\varepsilon}$ and positive integer $Q\ge Q_0(A,\varepsilon)$  we have
     $$
     N_A(Q, \delta,\bm{\theta})\le C_{d+1,m}\,\delta^mQ^{d+1}.
     $$
   Or equivalently, for $\delta> 0$  and positive integer $Q$ we have
      $$
     N_A(Q, \delta,\bm{\theta})\le \max\Big\{C_{d+1,m}\delta^mQ^{d+1},\,C(A,\varepsilon)Q^{d+1-\frac{m}\omega+\varepsilon}\Big\}.
     $$

     \item
     Suppose that $\omega(A')=\omega$. Then  for $\delta\ge Q^{-\frac1\omega+\varepsilon}$  and   positive integer $Q\ge Q_0(A',\varepsilon)$ we have
     $$
     N_{A}'(Q, \delta,\bm{\theta})\le C_{d,m}\,\delta^mQ^{d}.
     $$
 \end{enumerate}
  Here $Q_0(A,\varepsilon)$ and $C(A,\varepsilon)$ are positive constants  depending on $A$ and $\varepsilon$.
 \end{thm}
 
 In Theorem \ref{UpperBound}(a), we establish an analogue of Conjecture \ref{MainConjecture} for affine subspaces. The only difference here is that the range of $\delta$ should depend on the diophantine exponent of $A$. Moreover, we believe that the upper bound   $N_A(Q, \delta,\bm{\theta})\ll\delta^mQ^{d+1}$ no longer holds when $\delta$ drops below the critical threshold $Q^{-\frac1\omega}$. 
 
 Previously, the special case when $A\in\R^{1\times n}$ (in this case $\mL_A$ is a point) has been proved in \cite[Lemma 7]{RSS17} by geometry of numbers. The general problem has  been treated earlier in \cite{HL21} using  assumptions on the semi-multiplicative exponent $\omega^\times$ which is stronger than those assumed in Theorem \ref{UpperBound}, and yet the result therein is considerably weaker. 
 
 Our next theorem is a companion result to Theorem \ref{UpperBound}, using the  exponent  $\sigma(A)=\omega(A^\mathrm{T})$ in lieu of $\omega(A)$. 
 \begin{thm}\label{DualUpperBound}
  Let  $A$ be a $(d+1)\times m$ real matrix with $\sigma(A)=\sigma<\infty$, $\bm{\theta}\in\R^{d+m}$, and $\varepsilon$ be a positive number. Then there is some positive constant $C^*(A,\varepsilon)$ depending on $A$ and $\varepsilon$ such that  for $Q\in\N$ and $\delta>0$ we have
     $$
     N_A(Q, \delta,\bm{\theta})\le \max\Big\{\pi^{2m},\,C^*(A,\varepsilon)\delta^m Q^{m\sigma+\varepsilon}\Big\}.
     $$
 \end{thm}
 
We may compare the above two theorems and discuss their strengths and weaknesses. Theorem \ref{UpperBound} is good at capturing the heuristic main term $\delta^mQ^{d+1}$ when $\delta$ is larger than $Q^{-\frac1\omega}$, but it delivers inferior bounds when $\delta$ is much smaller. On the other hand, Theorem \ref{DualUpperBound} initially gives bad bounds for large $\delta$, but for smaller $\delta$ the bound becomes significantly better. Especially when $\delta$ is below $Q^{-{\sigma}}$ it infers that the counting function is bounded. This is an incredibly good bound since there may be finitely many rational points lying on $\mL$ and no matter how small $\delta$ is these points are always counted.

 The main innovation of this paper is that we manage to incorporate the multidimensional large sieve inequality and its dual form into Fourier analytic methods, which turns out to be  surprisingly effective in delivering  sharp upper bounds for our counting functions. The large sieve together with its relatives is a very powerful tool in analytic number theorists' repertoire. We refer the readers to Montgomery's excellent survey article  \cite{Mon78} which summarizes the development of the large sieve from its birth in the 1940s up to the end of 1970s. 
 
 Lastly, we establish a matching lower bound. Actually for the purpose of proving the divergence part of the Khintchine-Jarn\'{i}k type theorem (Theorem \ref{JTD}), we need the ubiquity version of the lower bound estimate.

For  $0\le \kappa<1$, let
\begin{equation}\label{RKappa}
\mathcal{R}_A^\kappa(Q,\delta):=\left\{(q,\bm{a})\in \N\times \Z^d\biggm|
\begin{array}{c}
\bm{a}/q\in\mathcal{U}, \,\kappa Q<q\le Q,\\ \|(q,\bm{a})A\|<\delta
\end{array}\right\}.
\end{equation}
We also use  $\mu_d$ to denote the Lebesgue measure on $\R^d$.

 \begin{thm}\label{UbiquityLowerBound}
 Let $A$ be a $(d+1)\times m$ real matrix with diophantine exponent $\omega(A)=\omega$,  $\tau_0$ be defined in \eqref{tau0}, and $\varepsilon$ be a  positive number less than $\tau_0$. Then for any ball $\mathcal{B}\subseteq\mathcal{U}$ there exists $\kappa>0$ depending on $\mathcal{B}$ such that for
 $\delta$  satisfying
 $1\ge\delta\ge Q^{-\tau_0+\varepsilon}$
 and positive integer $Q\ge Q_0(A,\varepsilon, \mathcal{B})$, we have
 $$
 \mu_d\left(\mathcal{B}\cap \bigcup_{(q,\bm{a})\in\mathcal{R}_A^\kappa(Q,\delta)}B\left(\frac{\bm{a}}q,\frac{\kappa^{-1}}{Q^{\frac{d+1}d}\delta^{\frac{m}d}}\right)\right)\ge\frac12\mu_d(\mathcal{B}).
 $$

 \end{thm}

  After upper bounding the measure  on the left hand side of the above inequality in a trivial manner,  we obtain a lower bound for $\#\mathcal{R}_A^\kappa(Q,\delta)$. Then a combination of this with Theorem \ref{UpperBound}  yields  matching upper and lower bounds for the homogeneous counting function $N_A(Q,\delta):=N_A(Q,\delta,\bm{0})$ when $\delta$ is not too small.

 \begin{cor}\label{AsympMainTerm}  Let $A$ be a $(d+1)\times m$ real matrix with diophantine exponent $\omega(A)=\omega$. Then for all
 $\delta$  satisfying
 $1\ge\delta\ge Q^{-\tau_0+\varepsilon}$ for some $\varepsilon>0$, where $\tau_0$ is defined in \eqref{tau0},
 and for all sufficiently large $Q$, we have
 $$
 N_A(Q,\delta)\asymp_{A} \delta^mQ^{d+1}.
 $$
 \end{cor}

Throughout the paper, we will use Vinogradov's notation $f(x)\ll g(x)$ and Landau's notation $f(x)=O(g(x))$ to mean there exists a positive constant $C$ such that $|f(x)|\le Cg(x)$ as $x\to\infty$. Moreover, $f\asymp g$ means $f\ll g$ and $f\gg g$. 

 Theorems  \ref{UpperBound-Phi}, \ref{UpperBound}, and \ref{DualUpperBound} (upper bound estimates) are proved in \S\ref{ProofUpperBounds}, while the proof of Theorem \ref{UbiquityLowerBound} (ubiquitous lower bound) is presented in \S\ref{ProofLowerBound}.

\subsection{The upper bound estimates}\label{ProofUpperBounds}

\subsubsection{Some preliminary lemmata}
To prepare us for the proofs of the upper bound estimates in \S\ref{CountingResults}, we state a few lemmata first.

Denote $e^{2\pi i x}$ by $e(x)$.
The following is a multidimensional large sieve inequality, due to Huxley \cite{Hux68}. 
\begin{lem}[Huxley, 1968]\label{LargeSieve}
For $\bm{y}\in\mathbb{R}^k$, let 
$$S(\bm{y}) = \sum_{\bm{l}} c(\bm{l})e(\bm{l}\cdot \bm{y}),$$
where the summation is over integer points $\bm{l}=(l_1,\ldots, l_k)$ in the $k$ dimensional rectangle
$$N_i < l_i \le N_i+L_i \quad(i = 1,\dots, k),$$
and $c(\bm{l})$ are some complex coefficients depending on $\bm{l}$.
Let $\bm{y}^{(r)}=(y^{(r)}_1,\ldots,y^{(r)}_k)\in\mathbb{R}^k$ for $1\le r\le R$ satisfy that
$$
\max_{i}\lambda_i^{-1}\|y^{(r)}_i-y^{(s)}_i\|\ge1
$$
whenever $r\not=s$, where $\lambda_1, \ldots, \lambda_k$ are positive numbers not exceeding $\frac12$. Then
$$
\sum_{r=1}^R|S(\bm{y}^{(r)})|^2\le\left(\prod_{i=1}^k \left(L_i^{\frac12}+\lambda_i^{-\frac12}\right)^2\right)\sum_{\bm{l}} |c(\bm{l})|^2.
$$
\end{lem}

Applying the duality principle of matrix norms (see for example \cite[Lemma 2]{Mon78}) to Lemma \ref{LargeSieve},  we obtain its dual form.
\begin{lem}[large sieve, dual form]\label{DualLargeSieve}
Let $\bm{y}^{(r)}=(y^{(r)}_1,\ldots,y^{(r)}_k)\in\mathbb{R}^k$ for $1\le r\le R$ satisfy that
$$
\max_{i}\lambda_i^{-1}\|y^{(r)}_i-y^{(s)}_i\|\ge1
$$
whenever $r\not=s$, where $\lambda_1, \ldots, \lambda_k$ are positive numbers not exceeding $\frac12$, and let 
$$T(\bm{l}) = \sum_{r=1}^R d(\bm{y}^{(r)})e(\bm{l}\cdot \bm{y}^{(r)})$$
for some complex coefficients $d(\bm{y}^{(r)})$ depending on $\bm{y}^{(r)}$.
Then
$$
\sum_{\bm{l}}|T(\bm{l})|^2\le \left(\prod_{i=1}^k \left(L_i^{\frac12}+\lambda_i^{-\frac12}\right)^2\right)\sum_{r=1}^R |d(\bm{y}^{(r)})|^2,
$$
where  the summation on the left side is over integer points $\bm{l}=(l_1,\ldots, l_k)$ in the $k$ dimensional rectangle
$$N_i < l_i \le N_i + L_i \quad(i = 1,\dots, k).$$
\end{lem}

\begin{lem}\label{PsiBadAppr}
Suppose that a matrix $A$ is not $\phi$-approximable, then it is $c_0\phi$-badly approximable for some positive constant $c_0\le1$ depending on $A$ and $\phi$. 
\end{lem}

\begin{proof}
By definition, we know
$\|A\bm{q}^\mathrm{T}\|<\phi(|\bm{q}|)$ only has finitely many solutions in integer vectors $\bm{q}$. Clearly for each nonzero integer vector $\bm{q}$, we have $A\bm{q}\not=\bm{0}$. Otherwise $A(k\bm{q})^\mathrm{T}=\bm{0}$ for all $k\in\Z$, which contradicts with the assumption that $A$ is not $\phi$-approximable. Therefore, we may choose a small enough positive constant $c_0\le 1$ to cope with these finitely many exceptions, and guarantee that $\|A\bm{q}^\mathrm{T}\|\ge c_0\phi(|\bm{q}|)$ holds for all nonzero integer vectors $\bm{q}$.
\end{proof}

\begin{lem}\label{Fejer}
For a real number $\delta$ with $0<\delta\le\frac12$, 
let $\rchi_{\delta}(\Theta)$ be the characteristic function of the set $\{\Theta\in\R:\|\Theta\|\le\delta\}$
and consider the Fej\'er kernel
$$
\mathcal{F}_J(\Theta)=J^{-2}\left|\sum_{j=1}^Je(j\Theta)\right|^2,\quad\text{where }J=\left\lfloor\frac1{2\delta}\right\rfloor.
$$
Then for all $\Theta\in\R$
$$
\rchi_\delta(\Theta)\le \frac{\pi^2}4\mathcal{F}_J(\Theta).$$
\end{lem}
\begin{proof}
First of all, when $\Theta\in\Z$, we have $\rchi_\delta(\Theta)=\mathcal{F}_J(\Theta)=1$. The desired inequality holds trivially. So we may assume that $\Theta\not\in\Z$.
By geometric summation, we see that
$$
\left|\sum_{j=1}^Je(j\Theta)\right|=\left|\frac{e(\Theta)(1-e(J\Theta))}{1-e(\Theta)}\right|=\left|\frac{e(\frac{J\Theta}2)-e(-\frac{J\Theta}2)}{e(\frac\Theta2)-e(-\frac\Theta2)}\right|=\left|\frac{\sin(\pi J \Theta)}{\sin(\pi\Theta)}\right|,
$$
and therefore
$$\mathcal{F}_J(\Theta)=
\left(\frac{\sin(\pi J \Theta)}{J\sin(\pi\Theta)}\right)^2.$$
Then we split into two cases: $\|\Theta\|\le\delta$ and $\|\Theta\|>\delta$. Suppose first that $\|\Theta\|\le\delta$.  Then  in view of the inequality $$J\|\Theta\|\le\frac1{2\delta}\delta=\frac12,$$ we see that $$\|J\Theta\|=J\|\Theta\|.$$  
Also note that
$$\pi\|x\|\ge\sin (\pi \|x\|)\ge2\|x\|. $$ 
Hence we have $$\left(\frac{\sin(\pi J \Theta)}{J\sin(\pi\Theta)}\right)^2=\left(\frac{\sin(\pi \|J \Theta\|)}{J\sin(\pi\|\Theta\|)}\right)^2\ge\left(\frac{2 \|J \Theta\|}{J\pi\|\Theta\|}\right)^2=\left(\frac{2J\| \Theta\|}{J\pi\|\Theta\|}\right)^2=\frac4{\pi^2}.$$
In the latter case when $\|\Theta\|>\delta$, we simply have
$$
\mathcal{F}_J(\Theta)\ge0=\rchi_\delta(\Theta).
$$
In either case, it follows that
$$\rchi_\delta(\Theta)\le \frac{\pi^2}4\mathcal{F}_J(\Theta).$$
\end{proof}

\subsubsection{Proof of Theorem \ref{UpperBound-Phi}}

Recall that

$$N_A(Q, \delta,\bm{\theta}) := \#\left\{(q,\bm{a})\in\Z^{d+1}:  \|(q, \bm{a}+\bm{\theta}^{(1)})A-\bm{\theta}^{(2)}\| < \delta,\, |(q,\bm{a})|< Q \right\}$$
and
$$N_{A}'(Q, \delta,\bm{\theta}) :=\sup_{q\in\R} \#\left\{\bm{a}\in\Z^d:  \left\|(q, \bm{a}+\bm{\theta}^{(1)})A-\bm{\theta}^{(2)}\right\| < \delta,\, |\bm{a}|< Q \right\}$$
where
$$
\bm{\theta}=(\theta_1,\ldots,\theta_d,\theta_{d+1},\ldots,\theta_{d+m})=(\bm{\theta}^{(1)}, \bm{\theta}^{(2)})\in\R^{d}\times\R^{m}.
$$

We first observe that the theorem is trivial when $\delta>\frac12$, since in this case
$$
N_A(Q, \delta,\bm{\theta})=(2Q-1)^{d+1}<(2\delta)^m(2Q)^{d+1}<C_{d+1,m}\,\delta^mQ^{d+1},
$$
and
$$
N_{A}'(Q, \delta,\bm{\theta})=(2Q-1)^{d}<(2\delta)^m(2Q)^{d}<C_{d,m}\,\delta^mQ^{d}.
$$
So without loss of generality, we may assume that
$$
0<\delta\le\frac12.
$$

Next we apply the Fej\'{e}r kernel (Lemma \ref{Fejer}) to detect when the system of $m$ diophantine inequalities included in $\|(q, \bm{a}+\bm{\theta}^{(1)})A-\bm{\theta}^{(2)}\| < \delta$ holds. Precisely
\begin{align*}
N_A(Q, \delta,\bm{\theta})
&\le\sum_{\substack{(q,\bm{a})\in\Z^{d+1}\\
|(q,\bm{a})|< Q}}\prod_{v=1}^{m}\rchi_{\delta}\big((q, \bm{a}+\bm{\theta}^{(1)})\cdot\mathrm{col}_v(A)-\theta_{d+v}\big)\\
&\le\sum_{\substack{(q,\bm{a})\in\Z^{d+1}\\
|(q,\bm{a})|< Q}}\prod_{v = 1}^{m}\left(\frac{\pi^2}4\mathcal{F}_J\big((q, \bm{a}+\bm{\theta}^{(1)})\cdot\mathrm{col}_v(A)-\theta_{d+v}\big)\right)\\
&= \left(\frac{\pi^2}{4J^2}\right)^{m}\sum_{\substack{(q,\bm{a})\in\Z^{d+1}\\
|(q,\bm{a})|< Q}}\prod_{v = 1}^{m}\left|\sum_{j_v= 1}^J
e\Big(j_v\big((q, \bm{a}+\bm{\theta}^{(1)})\cdot\mathrm{col}_v(A)-\theta_{d+v}\big)\Big)\right|^2.
\end{align*}
The product over $v$ in the last line is
\begin{align*}
&\Bigg|\sum_{\substack{\bm{j}\in\N^{m}\\|\bm{j}| \le J}}
e\left(\sum_{v = 1}^{m}j_v\big((q, \bm{a}+\bm{\theta}^{(1)})\cdot\mathrm{col}_v(A)-\theta_{d+v}\big)\right)\Bigg|^2\\
=& \Bigg|\sum_{\substack{\bm{j}\in\N^{m}\\|\bm{j}| \le J}}c(\bm{j})e\left((q,\bm{a})A\bm{j}^\mathrm{T}\right)\Bigg|^2
\end{align*}
where
$$
c(\bm{j})=e\left(\bm{\theta}^{(1)}A'\bm{j}^\mathrm{T}-\bm{\theta}^{(2)}\cdot\bm{j}\right).
$$

Let $$P_{\bm{j}}:=A\bm{j}^\mathrm{T}.$$
Then when $\bm{j}_1, \bm{j}_2\in\N^{m}$ with $\bm{j}_1\neq \bm{j}_2$ and $|\bm{j}_1|, |\bm{j}_2|\le J$, we have $|\bm{j}_1-\bm{j}_2|< J$. Since $A$ is $\phi$-badly approximable and $\phi$ is decreasing, it follows that
$$
\|P_{\bm{j}_1}-P_{\bm{j}_2}\|=\|A(\bm{j}_1-\bm{j}_2)^\mathrm{T}\|\ge\phi(|\bm{j}_1-\bm{j}_2|)\ge\phi(J),
$$
which shows that the set of  points $\{P_{\bm{j}}\mid\bm{j}\in\N^{m},\, |\bm{j}| \le J\}$ is $\phi(J)$-separated modulo 1. Hence
we may apply the dual large sieve inequality (Lemma \ref{DualLargeSieve})  and obtain 

\begin{align*}
\sum_{\substack{(q,\bm{a})\in\Z^{d+1}\\
|(q,\bm{a})|< Q}}\Bigg|&\sum_{\substack{\bm{j}\in\N^{m}\\|\bm{j}| \le J}}c(\bm{j})e\left((q, \bm{a})A\bm{j}^\mathrm{T}\right)\Bigg|^2\\
\le&\left((2Q)^\frac12+\phi( J)^{-\frac12}\right)^{2(d+1)}\sum_{\substack{\bm{j}\in\N^{m}\\|\bm{j}| \le J}}|c(\bm{j})|^2\\
\quad\le&8^{d+1}\max\left\{Q,\,\phi( J)^{-1}\right\}^{d+1}J^{m},
\end{align*}
where in the last line we used 
the fact that $$|c(\bm{j})|=1.$$

Finally, on merging the above estimates and noting that
$$
\frac1{4\delta}< J=\left\lfloor \frac1{2\delta}\right\rfloor< \frac1\delta
$$
we arrive at the desired conclusion
\begin{align*}
N_A(Q, \delta,\bm{\theta})&\le\left(\frac{\pi^2}{4J^2}\right)^{m} 8^{d+1}J^{m}\max\left\{Q,\,\phi( J)^{-1}\right\}^{d+1}\\
&\le 8^{d+1}\pi^{2m}\delta^m\max\left\{Q,\,\phi( \delta^{-1})^{-1}\right\}^{d+1}.
\end{align*}

This proves part (a). To prove part (b), we follow the same argument. The only difference is that we will not make use of the first row of $A$ in the large sieve step, since we are not summing over $q$ anymore. This leads to a slightly different set of points and their corresponding coefficients. More precisely

\begin{align*}
&
\sum_{\substack{\bm{a}\in\Z^{d}\\
|\bm{a}|< Q}}\Bigg|\sum_{\substack{\bm{j}\in\N^{m}\\|\bm{j}| \le J}}
e\left(\sum_{v = 1}^{m}j_v\big((q, \bm{a}+\bm{\theta}^{(1)})\cdot\mathrm{col}_v(A)-\theta_{d+v}\big)\right)\Bigg|^2\\
=&\sum_{\substack{\bm{a}\in\Z^{d}\\
|\bm{a}|< Q}}\Bigg|\sum_{\substack{\bm{j}\in\N^{m}\\|\bm{j}| \le J}}c'(\bm{j})e\left(\bm{a}A'\bm{j}^\mathrm{T}\right)\Bigg|^2
\end{align*}
where
$$
c'(\bm{j})=e\left(q\bm{\beta}\cdot \bm{j}+\bm{\theta}^{(1)}A'\bm{j}^\mathrm{T}-\bm{\theta}^{(2)}\cdot\bm{j}\right).
$$
Now the set of points $$\left\{A'\bm{j}^\mathrm{T}\bigm|\bm{j}\in\N^{m}, \,|\bm{j}| \le J\right\}$$
 can again be shown to be $\phi(J)$-separated modulo 1, as long as $A'$ is $\phi$-badly approximable. Then we continue to apply the analytic large sieve, though this time it is $d$-dimensional instead of $d+1$. Note that the dependence on $q$ will disappear when applying $|c'(\bm{j})|=1$ after the large sieve step. Therefore the upper bound obtained this way is uniform in $q$.

The rest of the argument is exactly the same as the first part, and in order to avoid redundancy we leave it to the readers.

\subsubsection{Proof of Theorem \ref{UpperBound} modulo Theorem \ref{UpperBound-Phi}}
We will only prove (a), as the proof of (b) is very much analogous.  
 Since $\omega(A)=\omega$, by the definition of diophantine exponent
 $A$ is not $\phi$-approximable, where $\phi:t\to t^{-\omega-\varepsilon'}$ for any $\varepsilon'>0$. Then by Lemma \ref{PsiBadAppr}, $A$ is $c_0\phi$-badly approximable for some $1\ge c_0>0$ depending on $A$ and $\varepsilon'$.  Therefore by Theorem \ref{UpperBound-Phi}(a), we know that
\begin{equation}\label{OmegaUpperBound}
N_A(Q, \delta,\bm{\theta})\le C_{d+1,m}\, \delta^m\max\left\{Q,\,\frac1{c_0\delta^{\omega+\varepsilon'}}\right\}^{d+1}.
\end{equation}

To prove the first part of (a), let $\varepsilon'=\frac{\omega^2\varepsilon}2$.
Then it is readily checked that when $\delta\ge Q^{-\frac1\omega+\varepsilon}$ we have the inequality
$$
\frac1{c_0\delta^{\omega+\varepsilon'}}\le c_0^{-1}Q^{(\frac1\omega-\varepsilon)(\omega+\varepsilon')}\le c_0^{-1}Q^{1-\frac{\varepsilon\omega}2},
$$
which is less than $Q$ when $Q$ is sufficiently large. 

To prove the second part of (a), we  take $\varepsilon'=\frac{\omega\varepsilon}{d+1}$. It then follows from \eqref{OmegaUpperBound} that
     $$
     N_A(Q, \delta,\bm{\theta})\le C_{d+1,m}\max\left\{\delta^mQ^{d+1},\,c_0^{-d-1}\delta^{m-(d+1+\varepsilon)\omega}\right\}.
     $$
     Since the exponent of the $\delta$ in the second term is negative, we note that  for $\delta\ge\delta_0:= Q^{-\frac1\omega}$   this is further bounded by 
     $$
     C_{d+1,m}\max\left\{\delta^mQ^{d+1},\,c_0^{-d-1}Q^{-\frac{m}\omega+d+1+\varepsilon}\right\}.
     $$
 So it just remains to check, by the fact that $N_A(Q, \delta,\bm{\theta})$ is increasing in $\delta$ for fixed $Q$ and $\bm{\theta}$, that when $\delta< \delta_0$ we have
$$
N_A(Q, \delta,\bm{\theta})\le N_A(Q, \delta_0,\bm{\theta})\le C(A,\varepsilon)Q^{-\frac{m}\omega+d+1+\varepsilon},
$$
where 
$$
C(A,\varepsilon):=C_{d+1,m}\,c_0^{-d-1}.
$$
So in any case, we conclude for all $\delta>0$ and positive integer $Q$  that
$$
 N_A(Q, \delta,\bm{\theta})\le   \max\left\{  C_{d+1,m}\delta^mQ^{d+1},\,C(A,\varepsilon)Q^{-\frac{m}\omega+d+1+\varepsilon}\right\}.
     $$
 
\subsubsection{Proof of Theorem \ref{DualUpperBound}} Let us first assume that $0<\delta\le \frac12$.
We follow the proof of Theorem \ref{UpperBound-Phi} until we arrive at
\begin{align*}
N_A(Q, \delta,\bm{\theta})\le \left(\frac{\pi^2}{4J^2}\right)^{m}\sum_{\substack{(q,\bm{a})\in\Z^{d+1}\\|(q,\bm{a})|< Q}}\Bigg|\sum_{\substack{\bm{j}\in\N^m\\|\bm{j}| \le J}}c(\bm{j})e\left((q,\bm{a})A\bm{j}^\mathrm{T}\right)\Bigg|^2.
\end{align*}
Then note that
$$
e\left((q,\bm{a})A\bm{j}^\mathrm{T}\right)=e\left(\bm{j}A^\mathrm{T}(q,\bm{a})^\mathrm{T}\right),
$$
which suggests that we may treat $\{A^\mathrm{T}(q,\bm{a})^\mathrm{T}\}$ in lieu of $\{A\bm{j}^\mathrm{T}\}$ as a set of separated points modulo 1. Indeed, by Lemma \ref{PsiBadAppr} we see that  there is some constant $0<c_0\le 1$ depending on $A$ and $\varepsilon$ such that $A^\mathrm{T}$ is $\phi$-badly approximable with $\phi:t\to c
_0t^{-\sigma-\varepsilon/m}$. Hence the point set $\{A^\mathrm{T}(q,\bm{a})^\mathrm{T}:|(q,\bm{a})|<Q\}$ is $c_0(2Q)^{-\sigma-\varepsilon/m}$ separated modulo 1. Thus, an application of the large sieve inequality (Lemma \ref{LargeSieve}) reveals that
\begin{align*}
N_A(Q, \delta,\bm{\theta})\le&\left(\frac{\pi^2}{4J^2}\right)^{m}\left(J^\frac12+c_0^{-\frac12}(2Q)^{\frac{\sigma+\varepsilon/m}2}\right)^{2m}\sum_{\substack{\bm{j}\in\N^m\\|\bm{j}| \le J}}|c(\bm{j})|^2\\
\le&\pi^{2m}J^{-m}\max\left\{J,\,c_0^{-1}(2Q)^{\sigma+\varepsilon/m}\right\}^m\\
\left(J=\left\lfloor(2\delta)^{-1}\right\rfloor\right)\quad\le&\max\left\{\pi^{2m},\,C^*(A,\varepsilon)\delta^m Q^{m\sigma+\varepsilon}\right\},
\end{align*}
where 
$$C^*(A,\varepsilon):=\pi^{2m}2^{(\sigma+2)m+\varepsilon}c_0^{-m}$$
is a constant depending on $A$ and $\varepsilon$. 

Finally, we observe that this bound actually also holds for $\delta>\frac12$, since
$$
N_A(Q, \delta,\bm{\theta})<(2Q)^{d+1}\le (2\delta)^m(2Q)^{m\sigma}<C^*(A,\varepsilon)\delta^m Q^{m\sigma}
$$
in view of $\sigma\ge\frac{d+1}m$.

\subsection{The ubiquitous lower bound}\label{ProofLowerBound}

\subsubsection{A covering lemma via Minkowski's theorem}

\begin{lem}\label{Minkowski}
Suppose $Q\in\N$ and $ Q^{-\frac1m}\le\delta\le1$. Then
$$
\mathcal{U}\subseteq\bigcup_{(q,\bm{a})\in\mathcal{R}^0(Q,\delta)}B\left(\frac{\bm{a}}{q},\frac1{qQ^{\frac1d}\delta^{\frac{m}d}}\right),
$$
where $\mathcal{R}^\kappa(Q,\delta)$ is defined in \eqref{RKappa}.
\end{lem}
\begin{proof}
Write $A=\begin{pmatrix}\bm{\beta}\\A'\end{pmatrix}$ with $\bm{\beta}=(\beta_1,\ldots, \beta_m)$ and $A'=(\alpha_{ij})\in\R^{d\times m}$. 
For any $\bm{x}\in\mathcal{U}=[0,1]^d$, consider the linear system of diophantine inequalities in $(q,\bm{a},\bm{b})\in\Z^{n+1}$
$$
\left\{
\begin{aligned}
|q|&\le Q\\
|qx_1-a_1|&<\left(Q^{\frac1d}\delta^{\frac{m}d}\right)^{-1}\\
\vdots\\
|qx_d-a_d|&<\left(Q^{\frac1d}\delta^{\frac{m}d}\right)^{-1}\\
|q\beta_1+a_1\alpha_{11}+&\cdots+a_d\alpha_{d1}-b_1|<\delta\\
\vdots\\
|q\beta_m+a_1\alpha_{1m}+&\cdots+a_d\alpha_{dm}-b_m|<\delta.
\end{aligned}
\right.
$$
Note that the linear forms on the left hand side can be rewritten in the matrix form
$$
\begin{pmatrix}
1&\bm{x}&\bm{\beta}\\
\bm{0}&-I_d&A'\\
0&\bm{0}&-I_m
\end{pmatrix}^\mathrm{T}
\begin{pmatrix}
q&\bm{a}&\bm{b}
\end{pmatrix}^\mathrm{T},
$$
and that
$$
Q\left(Q^{\frac1d}\delta^{\frac{m}d}\right)^{-d}\delta^m=1.
$$
Therefore by Minkowski's theorem on linear forms, we can find a nonzero integer solution $(q,\bm{a},\bm{b})$ to the above system. 
Moreover, we claim that there is an integer solution with $q>0$. Otherwise by symmetry $q=0$.  Then it follows from the condition $Q^{-\frac1m}\le\delta$ that $\left(Q^\frac1d\delta^\frac{m}d\right)^{-1}\le 1$, which, together with $\delta\le 1$, would force $(\bm{a},\bm{b})=\bm{0}$. 

To conclude the proof, it remains to check that $\frac{\bm{a}}q\in\mathcal{U}$, i.e. $0\le a_i\le q$ for all $i=1,\ldots, d$. This immediately follows from 
$$
-1\le qx_i-\left(Q^{\frac1d}\delta^{\frac{m}d}\right)^{-1}< a_i< qx_i+\left(Q^{\frac1d}\delta^{\frac{m}d}\right)^{-1}\le q+1, \quad i=1,\ldots, d.
$$

\end{proof}

\subsubsection{Proof of Theorem \ref{UbiquityLowerBound} modulo Theorem \ref{UpperBound}}
In view of Lemma \ref{Minkowski},  it suffices to show that there exists $0<\kappa<1$ such that for $1\ge\delta\ge Q^{-\tau_0+\varepsilon}$ and for sufficiently  large $Q$ we have
$$ \mu_d\left(\mathcal{B}\cap \bigcup_{\substack{(q,\bm{a})\in\N\times\Z^{d}\\|(q,\bm{a})|\le \kappa Q\\\|(q,\bm{a})A\|<\delta}}B\left(\frac{\bm{a}}q,\frac{1}{qQ^{\frac{1}d}\delta^{\frac{m}d}}\right)\right)
 \le
 \frac12\mu_d(\mathcal{B}).
$$ 
Without loss of generality, we may assume $Q\ge1/\kappa$, since otherwise there is nothing to prove.
Then splitting the summation over $|(q,\bm{a})|\le\kappa Q$ into dyadic ranges $2^k\le|(q,\bm{a})|<2^{k+1}$ and applying Theorem \ref{UpperBound}(a) to each of them, the left hand side is bounded by
\begin{align*}
    \sum_{\substack{(q,\bm{a})\in\N\times\Z^{d}\\|(q,\bm{a})|\le \kappa Q\\\|(q,\bm{a})A\|<\delta}}\frac{2^d}{q^dQ\delta^m}&\le 
    \frac{2^d}{Q\delta^m}\sum_{k=0}^{\lfloor\log_2(\kappa Q)\rfloor}N_A(2^{k+1},\delta)2^{-kd}\\[-20pt]
&\le\frac{2^d}{Q\delta^m} \sum_{k=0}^{\lfloor\log_2(\kappa Q)\rfloor}C_{d+1,m}\,\delta^m2^{k+d+1}
+C(A,\varepsilon)2^{(k+1)(1-\frac{m}\omega+\frac\varepsilon2)+d}\\
&\le4^{d+1}C_{d+1,m}\,\kappa+\frac{4^d}{Q\delta^m}C(A,\varepsilon)(2 Q)^{\max\{1-\frac{m}\omega,0\}+\frac\varepsilon2}(\log_2Q+1)\\
&\le4^{d+1}C_{d+1,m}\,\kappa+4^{d+1}C(A,\varepsilon)(\delta Q^{\tau_0})^{-m}Q^{\frac{\varepsilon}2}\log_2Q\\
(\delta Q^{\tau_0}\ge Q^\varepsilon)\quad&\le4^{d+1}C_{d+1,m}\,\kappa+4^{d+1}C(A,\varepsilon)Q^{-\frac{\varepsilon}2}\log_2Q.
\end{align*}
This can be made 
$$
\le \frac{\mu_d(\mathcal{B})}2
$$
when we choose 
$$\kappa=\frac{\mu_d(\mathcal{B})}{4^{d+2}C_{d+1,m}}
$$
and take Q sufficiently large, say $Q\ge Q_0(A,\varepsilon,\mathcal{B})$, so that the second term above is bounded by  $\frac14\mu_d(\mathcal{B})$.

\section{Diophantine approximation on affine subspaces}

\subsection{Two general Jarn\'{i}k type theorems for convergence}

Here we state and prove two general Jarn\'{i}k type theorems in the convergence aspect, of which Theorems \ref{CriticalKTC}, \ref{JTC} and \ref{StrongJTC} are consequences. 
 
\begin{thm}\label{JTC-Phi}  Let $\mathcal{L}$ be an affine subspace parametrized by $A$ as in  \eqref{AffineSubspace}, and let $0<s\le d$.  Suppose that $A$ is not $\phi$-approximable for a decreasing function  $\phi:\R_{>0}\to\R_{>0}$, and also
suppose that there exists a decreasing function
$\gamma:\N\to\R_{>0}$ satisfying the following two conditions:
\begin{equation}\label{GammaCondition}
\begin{aligned}
(i)&\quad\inf_{q\in\N} \phi\left(\gamma(q)^{-1}\right)q>0;
\\
 (ii)&  \quad \sum_{q=1}^\infty\gamma(q)^{m+ s}q^{d - s}<\infty.
\end{aligned}
\end{equation}
Then for any decreasing function $\psi:\N\to\R_{\ge0}$ such that the convergence condition \eqref{ConvCondition} holds,
and for any $\bm{\theta}\in\R^n$, we have
$$\mathcal{H}^s(\mathcal{S}_n(\psi,\bm{\theta}) \cap \mathcal{L}) = 0.
$$
\end{thm}

\begin{thm}\label{StrongJTC-Phi}   
Let $\mathcal{L}$ be an affine subspace parametrized by $A$ as in  \eqref{AffineSubspace}, $\mathcal{Q}$ be an infinite subset of $\N$, and let $0<s\le d$. Suppose that $A'\in\R^{d\times m}$ is not $\phi$-approximable for a decreasing function  $\phi:\R_{>0}\to\R_{>0}$,  and also suppose  that there exists a decreasing function
$\gamma:\mathcal{Q}\to\R_{>0}$ satisfying the following two conditions:
\begin{equation}\label{GammaCondition'}
\begin{aligned}
(i)&\quad\inf_{q\in\mathcal{Q}} \phi\left(\gamma(q)^{-1}\right)q>0;\\
(ii)&\quad    \sum_{q \in\mathcal{Q}}\gamma(q)^{m+ s}q^{d - s}<\infty.
    \end{aligned}
\end{equation}
Then for any (not necessarily decreasing) function $\psi:\mathcal{Q}\to\R_{\ge0}$ such that the convergence condition \eqref{ConvCondition-Strong} holds,
and for any $\bm{\theta}\in\R^n$, we have$$\mathcal{H}^s(\mathcal{S}_n(\psi,\bm{\theta}) \cap \mathcal{L}) = 0.
$$
\end{thm}

\subsubsection{Proof of Theorem \ref{CriticalKTC} modulo Theorem \ref{JTC-Phi}}

\begin{lem}\label{WellAppr}
Let $\mathcal{L}$ be an affine subspace parametrized by $A$ as in  \eqref{AffineSubspace}. For real numbers $a>0$, $a'$, $b$, and $b'$ satisfying
$$
b=\frac{a}{(n-1)a+n},\quad b'<\frac{a'}{(n-1)a+n},
$$
define the following two functions:
$$
\phi:\;t\to t^{-a}(\log t)^{-a'},
$$
and
$$
\psi:\;q\to q^{-b}(\log q)^{-b'}.
$$
Suppose $A$ is $\phi$-approximable.
Then 
$$
\mathcal{L}\subset\mathcal{S}_n(\psi),
$$
that is,
all points of $\mathcal{L}$ are simultaneously $\psi$-approximable.
\end{lem}
\begin{proof}
Since $A$ is $\phi$-approximable, there are infinitely many $\bm{q}\in\Z^m$ such that
$$
\left|A\bm{q}^\mathrm{T}+\bm{p}^\mathrm{T}\right|<\phi(|\bm{q}|)
$$
holds for some $\bm{p}=(p_0,p_1,\ldots, p_d)\in\Z^{d+1}$. Write
$\bm{p}'$ for $(p_1,\ldots, p_d)$. Then for any $\bm{x}\in\R^d$, we have
\begin{align*}
\left|p_0+(\bm{x},\widetilde{\bm{x}}A)(\bm{p}',\bm{q})^\mathrm{T}\right|&=\left|p_0+\bm{x}{\bm{p}'}^\mathrm{T}+\widetilde{\bm{x}}A\bm{q}^\mathrm{T}\right|=\left|\widetilde{\bm{x}}\left(A\bm{q}^\mathrm{T}+\bm{p}^\mathrm{T}\right)\right|\\
&\le (d+1)|\widetilde{\bm{x}}|\left|A\bm{q}^\mathrm{T}+\bm{p}^\mathrm{T}\right|<(d+1)|\widetilde{\bm{x}}|\phi(|\bm{q}|).
\end{align*}
Let $\varepsilon$ be a small positive number. Thus it is easy to see that
$$
\left\|(\bm{x},\widetilde{\bm{x}}A)(\bm{p}',\bm{q})^\mathrm{T}\right\|<|\bm{q}|^{-a}(\log |\bm{q}|)^{-a'+\varepsilon}
$$
holds for all but finitely many of the above $(\bm{p}',\bm{q})$.  Moreover, since $|\bm{p}|\ll_A |\bm{q}|$, the above inequality still holds with $|\bm{q}|$ replaced by $\left|(\bm{p}',\bm{q})\right|$ and $\varepsilon$ replaced by $2\varepsilon$ on the right hand side, after possibly excluding another finite set of $(\bm{p}',\bm{q})$. Then the lemma immediately follows after an application of Mahler's transference principle (see \cite[Chapter V Theorem 2]{Cas57}), which turns this dual approximation inequality into one about simultaneous approximation. 
\end{proof}

Now we embark on the proof of Theorem \ref{CriticalKTC}.
\begin{enumerate}[(a)]
\item Pick a small $\varepsilon>0$ such that $1+\varepsilon<-\omega'(A)$. Let
$$
\phi:\;t\to t^{-n}(\log t)^{1+\varepsilon},
$$
and
$$
\gamma:\;q\to q^{-\frac1n}(\log q)^{-\frac{1+\varepsilon}n}.
$$
Then by definition, $A$ is not $\phi$-approximable. Furthermore, one may readily check that both conditions in \eqref{GammaCondition} are met with $s=d$. So as a result of Theorem \ref{JTC-Phi}, $\mathcal{L}$ is of Khintchine type for convergence.
\item Select an $\varepsilon>0$ so that $n(1+\varepsilon)<\omega'(A)$, and let $a=n$ and $a'=n(1+\varepsilon)$. Then $A$ is $\phi$-approximable with $\phi(t)=t^{-a}(\log t)^{-a'}$. Therefore, by Lemma \ref{WellAppr} one can derive at once that
$$
\mathcal{L}\subset\mathcal{S}_n(\psi)\quad \text{with}\enskip\psi(q)=q^{-\frac1n}(\log q)^{-\frac{1+\varepsilon}n},
$$
which shows that $\mathcal{L}$ cannot be of Khintchine type for convergence. 
\end{enumerate}

\subsubsection{Proof of Theorems \ref{JTC} and \ref{StrongJTC} modulo Theorems \ref{JTC-Phi} and \ref{StrongJTC-Phi}}
 We only prove Theorem \ref{StrongJTC} using Theorem \ref{StrongJTC-Phi}, as the proof of Theorem \ref{JTC} using Theorem \ref{JTC-Phi} is almost identical (taking $\mathcal{Q}=\N$ and $\nu=1$). 

From \eqref{RangeOfS-Strong}, we deduce that
$$
\frac{d+\nu-s}{m+s}\omega<1.
$$
Hence we may find an $\varepsilon>0$ such that
$$
\left(\frac{d+\nu-s}{m+s}+\varepsilon\right)(\omega+\varepsilon)<1.
$$
Let
$$
\phi:\;t\to\frac1{t^{\omega+\varepsilon}}
$$
and
$$
\gamma:\;q\to q^{-\frac{d+\nu-s}{m+s}-\varepsilon}.
$$
By the definition of diophantine exponent, we know that $A'$ is not $\phi$-approximable. Moreover, it is readily verified that $\gamma$ satisfies both conditions in \eqref{GammaCondition'}, indeed
$$
\inf_{q\in\mathcal{Q}}\phi\left(\gamma(q)^{-1}\right)q=\inf_{q\in\mathcal{Q}}q^{1-\left(\frac{d+\nu-s}{m+s}+\varepsilon\right)(\omega+\varepsilon)}=\infty
$$
and
$$
\sum_{q \in\mathcal{Q}}\gamma(q)^{m+ s}q^{d - s}=\sum_{q\in\mathcal{Q}}q^{-\nu-\varepsilon(m+s)}<\infty.
$$
This shows that all the assumptions of Theorem \ref{StrongJTC-Phi} are met, hence Theorem \ref{StrongJTC} immediately follows.

\subsection{Reduction to the unit cube}\label{ReductionStep}
Suppose an affine subspace $\mathcal{L}$ is given by \eqref{AffineSubspace}. Let 
 $\Omega(A, \psi,\bm{\theta})$ denote the projection of $\mathcal{L}\cap \mathcal{S}_n(\psi,\bm{\theta})$ onto $\R^d$ via the map $(\bm{x},\widetilde{\bm{x}}A)\to\bm{x}$. Since this projection map  is bi-Lipschitz, we know that
\begin{equation}\label{Projection}
\begin{array}{c}
\mathcal{H}^s(\mathcal{L}\cap \mathcal{S}_n(\psi,\bm{\theta}))=0 \;(\text{resp.}\; \infty)\\ \Big\Updownarrow\\
\mathcal{H}^s(\Omega(A, \psi,\bm{\theta}))=0 \;\left(\text{resp.}\; \infty\right). 
\end{array}
\end{equation}
Moreover when $s=d$, the set $\mathcal{L}\cap \mathcal{S}_n(\psi,\bm{\theta})$ has full measure in $\mathcal{L}$ if and only if $\Omega(A, \psi,\bm{\theta})$ has full measure in $\R^d$. In the following we are going to show that it suffices to prove the desired theorems in \S\ref{DioExpExtremal} and \S\ref{KhintchineJarnik} by restricting $\Omega(A, \psi,\bm{\theta})$ to the unit cube $\mathcal{U}=[0,1]^d$.

\begin{lem} \label{Shift}
For $\bm{k}\in \Z^d$, let
$$
\mathcal{U}_{\bm{k}}=\mathcal{U}+\bm{k}.
$$
Then
$$
\Omega(A, \psi,\bm{\theta})\cap \mathcal{U}_{\bm{k}}=\Omega(A_{\bm{k}},\psi,\bm{\theta})\cap \mathcal{U}+\bm{k},
$$
where
\begin{equation}\label{Ak}
A_{\bm{k}}=\begin{pmatrix}
1&\bm{k}\\
\bm{0}&I_d
\end{pmatrix}A.
\end{equation}
\end{lem}

\begin{proof}
For $\bm{x}\in\mathcal{U}_{\bm{k}}$, write $\bm{x}=\bm{y}+\bm{k}$
with $\bm{y}\in\mathcal{U}$. Then observe that
$$
\widetilde{\bm{x}}=(1,\bm{y}+\bm{k})=\widetilde{\bm{y}}\begin{pmatrix}
1&\bm{k}\\
\bm{0}&I_d
\end{pmatrix}.
$$
Hence
\begin{align*}
\Omega(A, \psi,\bm{\theta})\cap \mathcal{U}_{\bm{k}}&=\left\{\bm{x}\in\mathcal{U}_{\bm{k}}\bigm|(\bm{x},\widetilde{\bm{x}}A)\in\mathcal{S}_n(\psi,\bm{\theta})\right\}\\
&=\left\{\bm{y}\in\mathcal{U}\bigm|(\bm{y}+\bm{k},\widetilde{\bm{y}}A_{\bm{k}})\in\mathcal{S}_n(\psi,\bm{\theta})\right\}+\bm{k}.
\end{align*}
The upshot here is that $\mathcal{S}_n(\psi,\bm{\theta})$ is invariant under integer translation. So it follows that
\begin{align*}
\Omega(A, \psi,\bm{\theta})\cap \mathcal{U}_{\bm{k}}&=\left\{\bm{y}\in\mathcal{U}\bigm|(\bm{y},\widetilde{\bm{y}}A_{\bm{k}})\in\mathcal{S}_n(\psi,\bm{\theta})\right\}+\bm{k}\\
&=\Omega(A_{\bm{k}},\psi,\bm{\theta})\cap \mathcal{U}+\bm{k}.
\end{align*}
\end{proof}
The next lemma shows that the matrices $A_{\bm{k}}$ obtained in \eqref{Ak} still retain the diophantine properties of $A$.

\begin{lem}\label{Transform}
Let $A\in\R^{d\times m}$,  $D\in\Z^{d\times d}$ with $|\det D|=1$, and $B=DA$.

\begin{enumerate}[(a)]
\item If  $A\in\R^{d\times m}$ is  $\phi$-badly approximable, then  $B$ is $c_0\phi$-badly approximable for some positive constant $c_0$ depending on $D$;
\item If $A^\mathrm{T}$ is $\phi$-badly approximable, then  $B^\mathrm{T}$ is $\widetilde{\phi}$-badly approximable, where $\widetilde{\phi}:t\to\phi\left(\lambda t\right)$ for some $\lambda>0$;
\item $\omega(A)=\omega(B)$ and $\sigma(A)=\sigma(B)$.
\end{enumerate}
\end{lem}

\begin{proof}
(a) 
Since $\det D=\pm1$, we see that $D^{-1}\in\Z^{d\times d}$. Let  $c_0=(d| D^{-1}|)^{-1}$. Recall that $|D^{-1}|$ denotes the supremum norm of $D^{-1}$, that is, the maximum of absolute values of its entries. 
Suppose that there exists a nonzero vector $\bm{q}\in\Z^m$ such that
$$
\big\|B\bm{q}^\mathrm{T}\big\|<c_0 \phi(|\bm{q}|).
$$
Then
$$
\left\|A\bm{q}^\mathrm{T}\right\|=\big\| D^{-1}B\,\bm{q}^\mathrm{T}\big\|\le d| D^{-1}|\,\big\|B\bm{q}^\mathrm{T}\big\|<\phi(|\bm{q}|),
$$
which contradicts with the assumption that $A$ is $\phi$-badly approximable. So this proves that $B$ is $c_0\phi$-badly approximable. 
\medskip

\noindent(b) Let $\lambda:=d|D|$. It follows from $|D^\mathrm{T}\bm{q}|\le \lambda|\bm{q}|$ that
$$
\|B^\mathrm{T}\bm{q}\|=\|A^\mathrm{T}D^\mathrm{T}\bm{q}\|\ge\phi(|D^\mathrm{T}\bm{q}|)\ge\phi(\lambda|\bm{q}|).
$$
Hence $B^\mathrm{T}$ is $\widetilde{\phi}$-badly approximable with $\widetilde{\phi}:t\to\phi\left(\lambda t\right)$.

\medskip
\noindent(c) Since  the roles of $A$ and $B$ can be reversed in parts (a) and (b), we immediately see that $\omega(A)=\omega(B)$ and $\sigma(A)=\sigma(B)$ in view of Lemma \ref{PsiBadAppr} and the definition of diophantine exponent. 
\end{proof}

\subsection{The convergence aspect}
\subsubsection{Proof of Theorem \ref{JTC-Phi}}\label{ProofJTC-Phi}

 Our plan is to first of all show that
 \begin{equation}\label{ZeroMeasure}
 \mathcal{H}^s(\Omega(A, \psi,\bm{\theta})\cap \mathcal{U})=0
 \end{equation}
under the assumption that $A$ is not $\phi$-approximable and $\psi$ is decreasing. 

Let $\Xi\left(\frac{\bm{p}+\bm{\theta}}q\right)$ denote the set of $\bm{x}\in \mathcal{U}$ satisfying the system of inequalities
\begin{equation}\label{Xi}
\begin{cases}
\big|\bm{x}-\frac{\bm{a}+\bm{\theta}^{(1)}}q\big|<\frac{\psi(q)}q\\
\big|\bm{\widetilde{x}}A-\frac{\bm{b}+\bm{\theta}^{(2)}}q\big|<\frac{\psi(q)}q
\end{cases}
\end{equation}
where $\bm{p}=(\bm{a},\bm{b})\in\Z^d\times\Z^{m}$. Clearly 
\begin{equation}\label{limsup}
\Omega(A, \psi,\bm{\theta})\cap \mathcal{U}=\limsup_{(q,\bm{p})\in\N\times\Z^n}\Xi\left(\frac{\bm{p}+\bm{\theta}}q\right).
\end{equation}

Assuming $\Xi\left(\frac{\bm{p}+\bm{\theta}}q\right)\neq \emptyset$,  we will make a few observations.
\begin{enumerate}[(i)]
\item
We have by the triangle inequality that
\begin{equation}
\begin{aligned}
&\left|\left(1,\frac{\bm{a}+\bm{\theta}^{(1)}}q\right)A-\frac{\bm{b}+\bm{\theta}^{(2)}}q\right|\\
\le&\left|\left(1,\frac{\bm{a}+\bm{\theta}^{(1)}}q\right)A-\bm{\widetilde{x}}A\right|+\left|\bm{\widetilde{x}}A-\frac{\bm{b}+\bm{\theta}^{(2)}}q\right|\\
\overset{\eqref{Xi}}{<}& \frac{c_1\psi(q)}{q}
\end{aligned}\label{NearIneq1}
\end{equation}
where 
$c_1:=d|A'|+1$
is a positive constant depending only on $A'$. Then it follows that
\begin{equation}\label{NearIneq2}
\left\|\left(q,\bm{a}+\bm{\theta}^{(1)}\right) A-\bm{\theta}^{(2)}\right\| < c_1\psi(q).
\end{equation}
\item
Because of the convergence condition \eqref{ConvCondition}, $\psi(q)\to0$ as $q\to\infty$, hence we may assume $\psi(q)\le \frac1{2c_1}$ for all sufficiently large $q$ so that \eqref{NearIneq1} has exactly one solution in $\bm{b}$ for each $(q,\bm{a})$ satisfying \eqref{NearIneq2}.
\item
For sufficiently large $q$ we see by the triangle inequality again that
\begin{equation}\label{BoundForA}
\left|\frac{\bm{a}}q\right|<|\bm{x}|+\frac{|\bm{\theta}^{(1)}|}q+\frac{\psi(q)}q<2.
\end{equation}
\end{enumerate}

So from these discussions, we deduce for sufficiently large $k\in \N$ that
\begin{align*}
&\#\left\{(q,\bm{p})\in\N\times\Z^n\biggm|2^{k}\le q<2^{k+1},\,\Xi\left(\frac{\bm{p}+\bm{\theta}}q\right)\neq\emptyset\right\}\\
\le&\#\left\{(q,\bm{a})\in\N\times\Z^d\Bigm|2^{k}\le q<2^{k+1},\,|\bm{a}|<2q,\,\eqref{NearIneq2}\right\}\\
\le &\#\left\{(q,\bm{a})\in\Z^{d+1}\Bigm||(q,\bm{a})|<2^{k+2}, \,\left\|\left(q,\bm{a}+\bm{\theta}^{(1)}\right) A-\bm{\theta}^{(2)}\right\| < c_1\psi(2^k)\right\}\\
=&N_A\left(2^{k+2},c_1\psi(2^k),\bm{\theta}\right),
\end{align*}
where in the third line above we have used the assumption that $\psi$ is decreasing in order to pass from $\psi(q)$ with $q\ge2^k$ to $\psi(2^k)$. 
Thus for each sufficiently large dyadic block $2^k\le q<2^{k+1}$, there are at most $N_A\left(2^{k+2},c_1\psi(2^k),\bm{\theta}\right)$ nonempty $\Xi\left(\frac{\bm{p}+\bm{\theta}}q\right)$, and each of them is contained in a hypercube with side length $\frac{2\psi(2^k)}{2^k}$.

Now by Lemma \ref{PsiBadAppr}, $A$ is $c_0\phi$-badly approximable for some positive constant $c_0$. Then it follows from Theorem \ref{UpperBound-Phi}(a) that
$$
N_A\left(2^{k+2},c_1\psi(2^k),\bm{\theta}\right)\\
\le C_{d+1,m}\,c_1^m \psi(2^{k})^m\max\left\{2^{k+2},\,c_0^{-1}\phi\left( \frac1{c_1\psi(2^k)}\right)^{-1}\right\}^{d+1}.
$$

Next we show that there is no loss of generality in assuming that
\begin{equation}\label{Reduction}
\psi(q)\ge\gamma(q)\quad\text{for all } q\in\N.
\end{equation}
 Otherwise, we consider the auxiliary function
$$
\widetilde{\psi}(q):=\max\{\psi(q),\gamma(q)\}.
$$
Clearly $\widetilde{\psi}$ is a decreasing function satisfying the convergence condition $$\sum\widetilde{\psi}(q)^{m+s}q^{d-s}<\infty$$  since both $\psi$ and $\gamma$ are. Moreover in view of the relation $\widetilde{\psi}(q)\ge\psi(q)$, it follows by definition that $\mathcal{S}_n(\psi)\subseteq\mathcal{S}_n(\widetilde{\psi})$. So it suffices to prove the theorem with $\psi$ replaced by $\widetilde{\psi}$. Hence, we may proceed with the proof assuming \eqref{Reduction}. 

Since $c_1\ge1$ and $\phi$ is decreasing, it is readily verified that
$$
\phi\left( \frac1{c_1\psi(2^k)}\right)\overset{\eqref{Reduction}}{\ge} \phi\left( \frac1{\gamma(2^k)}\right)\overset{\eqref{GammaCondition}(i)}{\gg}\frac1{2^k},
$$
and consequently
$$
N_A\left(2^{k+2},c_1\psi(2^k),\bm{\theta}\right)\ll \psi(2^{k})^m(2^k)^{d+1}.
$$
Equipped with this crucial estimate,  we can easily check that
\begin{align*}
&\sum_{k=1}^\infty N_A\left(2^{k+2},c_1\psi(2^k),\bm{\theta}\right)\left(\frac{2\psi(2^k)}{2^k}\right)^s\\
\ll&\sum_{k = 1}^\infty\psi(2^k)^m(2^k)^{(d+1)}\left(\frac{\psi(2^k)}{2^k}\right)^s\\
\ll&\sum_{k = 1}^\infty\psi(2^k)^{m + s}(2^k)^{(d - s+1)}
\end{align*}
which is convergent in view of \eqref{ConvCondition} and the assumption that $\psi$ is decreasing. 
Then 
by the following Hausdorff-Cantelli lemma, \eqref{ZeroMeasure} follows immediately from \eqref{limsup}.
\begin{lem}[{\cite[Lemma 3.10]{BD99}}]\label{HC}
Let $\{H_k\}$ be a sequence of hypercubes in $\R^n$ with side lengths $\ell(H_k)$ and suppose that for some $s>0$,
$$
\sum_k\ell(H_k)^s<\infty,
$$
then $\mathcal{H}^s(\limsup H_k)=0$.
\end{lem}

Finally we are poised to demonstrate $\mathcal{H}^s(\Omega(A, \psi,\bm{\theta}))=0$ by showing that $\mathcal{H}^s(\Omega(A, \psi,\bm{\theta})\cap\mathcal{U}_{\bm{k}})$=0 for each $\bm{k}\in \Z^d$. Indeed, by Lemma \ref{Shift}, we have
$$
\mathcal{H}^s\left(\Omega(A, \psi,\bm{\theta})\cap \mathcal{U}_{\bm{k}}\right)=\mathcal{H}^s\left(\Omega(A_{\bm{k}},\psi,\bm{\theta})\cap \mathcal{U}\right),
$$
where $A_{\bm{k}}$ is defined in \eqref{Ak}. Here the point is that we may apply Lemma \ref{Transform}(a) to conclude that there is some positive constant $c_0$ such that each $A_{\bm{k}}$ is $c_0\phi$-badly approximable. Then by our argument \eqref{ZeroMeasure} still holds with $A$ replaced by $A_{\bm{k}}$, which implies that
$$
\mathcal{H}^s(\Omega(A, \psi,\bm{\theta}))=\sum_{\bm{k}\in\Z^d}\mathcal{H}^s\left(\Omega(A, \psi,\bm{\theta})\cap \mathcal{U}_{\bm{k}}\right)=0.
$$
Thus $\mathcal{H}^s(\mathcal{L}\cap \mathcal{S}_n(\psi,\bm{\theta}))=0$ in view of \eqref{Projection}, which therefore concludes the proof of  Theorem \ref{JTC-Phi}. 

\subsubsection{Proof of Theorem \ref{StrongJTC-Phi}}
The proof almost works in the same manner as that of Theorem \ref{JTC-Phi}.  In the following
we will only point out the necessary modifications, since the bulk of the proof remains unchanged.    

Recall that here we assume that $A'$ is not $\phi$-approximable but we no longer assume $\psi$ is monotonic. 
The only main difference this causes is that when counting the number of nonempty sets $\Xi\left(\frac{\bm{p}+\bm{\theta}}q\right)$, we can no longer partition $q$ into dyadic ranges; instead we will invoke the stronger individual bound for each $q$ established in Theorem \ref{UpperBound-Phi}(b),
since in this case $A'$ is $c_0\phi$-badly approximable for some $c_0>0$ by Lemma \ref{PsiBadAppr}. More precisely, it follows from \eqref{NearIneq2} and \eqref{BoundForA} that for sufficiently large $q\in\mathcal{Q}$ we have
\begin{align*}
    &\#\left\{\bm{p}\in\Z^n\biggm|\Xi\left(\frac{\bm{p}+\bm{\theta}}q\right)\neq\emptyset\right\}\\
    \le &\#\left\{\bm{a}\in\Z^d\Bigm||\bm{a}|<2q, \left\|\left(q,\bm{a}+\bm{\theta}^{(1)}\right) A-\bm{\theta}^{(2)}\right\| < c_1\psi(q)\right\}\\
=&N_A'(2q,c_1\psi(q),\bm{\theta}).
\end{align*}
For the same reasoning as in the proof of Theorem \ref{JTC-Phi}, we may assume without loss of generality that
\begin{equation}\label{Reduction'}
\psi(q)\ge\gamma(q)\quad\text{for all } q\in\mathcal{Q}.
\end{equation}
Hence for $q\in\mathcal{Q}$
$$
\phi\left( \frac1{c_1\psi(q)}\right)\overset{\eqref{Reduction'}}{\ge} \phi\left( \frac1{\gamma(q)}\right)\overset{\eqref{GammaCondition'}(i)}{\gg}\frac1{q}.
$$
Plugging in this estimate into Theorem \ref{UpperBound-Phi}(b) leads to
$$
N_A'(2q,c_1\psi(q),\bm{\theta})\ll \psi(q)^mq^d.
$$
So
\begin{align*}
&\sum_{q\in\mathcal{Q}}\#\left\{\bm{p}\in\Z^n\biggm|\Xi\left(\frac{\bm{p}+\bm{\theta}}q\right)\neq\emptyset\right\}\left(\frac{2\psi(q)}q\right)^s\\
\ll&\sum_{q\in\mathcal{Q}}\psi(q)^{m+s}q^{d-s}
\end{align*}
which is convergent in view of \eqref{ConvCondition-Strong}. The rest of the argument is exactly the same as in the proof of Theorem \ref{JTC-Phi}.

\subsubsection{Proof of Theorem \ref{DiophantineExp}}\label{ProofDiophantineExp}
We will prove the contrapositive statement of  Theorem \ref{DiophantineExp}, i.e.
if
\begin{equation}\label{NuCondition1}
    \tau>\frac1d\left(1-\frac{m}{\max\{n,\omega\}}\right)=\max\left\{\frac1n,\,\frac1d\left(1-\frac{m}{\omega}\right)\right\}
\end{equation}
or
\begin{equation}\label{NuCondition2}
    \tau>\frac{m\sigma-d}n
\end{equation}
then
$$
|\mathcal{S}_n(\tau,\bm{\theta})\cap\mathcal{L}|_\mathcal{L}=0.
$$
In view of the reduction argument in \S \ref{ReductionStep}, this amounts to proving that
$$
\mu_d(\Omega(A,\psi_{\tau},\bm{\theta})\cap\mathcal{U})=0,
$$
where 
$$
\psi_\tau:\;q\to q^{-\tau}
$$
and $\Omega(A,\psi_{\tau},\bm{\theta})$ is given in \eqref{limsup}. 

Then we largely follow the proof of Theorem \ref{JTC-Phi}. In particular, in \S \ref{ProofJTC-Phi} we have shown that for sufficiently large $k\in\N$
\begin{align*}
&\#\left\{(q,\bm{p})\in\N\times\Z^n\biggm|2^{k}\le q<2^{k+1}, \,\Xi\left(\frac{\bm{p}+\bm{\theta}}q\right)\neq\emptyset\right\}\\
\le&N_A\left(2^{k+2},c_1\psi_{\tau}(2^k),\bm{\theta}\right).
\end{align*}
Suppose $\eqref{NuCondition1}$ holds. 
Let 
$$
\varepsilon=\frac12\left(\tau d-1+\frac{m}{\omega}\right),
$$
which is positive due to \eqref{NuCondition1}. Next
 by Theorem \ref{UpperBound}(a), we have
 $$
 N_A\left(2^{k+2},c_1\psi_{\tau}(2^k),\bm{\theta}\right)\ll\max\left\{(2^k)^{-\tau m+d+1},(2^k)^{d+1-\frac{m}{\omega}+\varepsilon}\right\}.
 $$
 For $2^k\le q<2^{k+1}$, note that
 $$
 \mu_d\left(\Xi\left(\frac{\bm{p}+\bm{\theta}}q\right)\right)\le \left(\frac{2\psi_{\tau}(2^k)}{2^k}\right)^d=2^d(2^k)^{-\tau d-d}.
 $$
  It just remains to check that the series
 \begin{align*}
    & \sum_{k=1}^\infty N_A\left(2^{k+2},c_1\psi_{\tau}(2^k),\bm{\theta}\right)(2^k)^{-(\tau+1)d}\\
    \ll&\sum_{k=1}^\infty\max\left\{(2^k)^{-\tau n+1},(2^k)^{-\tau d+1-\frac{m}{\omega}+\varepsilon}\right\}
 \end{align*}
 is indeed convergent under the assumption \eqref{NuCondition1}, as the desired conclusion follows immediately from this and the Borel-Cantelli lemma. 
 
On the other hand, suppose \eqref{NuCondition2} holds. Then we apply Theorem \ref{DualUpperBound} instead of Theorem \ref{UpperBound}, and deduce that
\begin{align*}
    & \sum_{k=1}^\infty N_A\left(2^{k+2},c_1\psi_{\tau}(2^k), \bm{\theta}\right)(2^k)^{-(\tau+1)d}\\
    \ll&\sum_{k=1}^\infty\max\left\{(2^k)^{-(\tau +1)d}, (2^k)^{-\tau m+m\sigma+\varepsilon-(\tau+1)d}\right\}
 \end{align*}
is again convergent on choosing a sufficiently small $\varepsilon>0$.

\subsubsection{Proof of Theorem \ref{Dimension}}\label{ProofDimension}

By Theorems \ref{UpperBound}(a) and \ref{DualUpperBound}, one notes that for any $\varepsilon>0$ and sufficiently large $Q$
$$
N_A(Q,\delta,\bm{\theta})\ll\left\{
\begin{aligned}
\delta^m Q^{d+1+\varepsilon} \quad& \text{when}\;\;\delta\ge Q^{-\frac1\omega};\\
Q^{d+1-\frac{m}{\omega}+\varepsilon}\quad&\text{when}\;\; \delta\le Q^{-\frac1\omega};\\
\delta^mQ^{m\sigma+\varepsilon}\quad&\text{when}\;\;\delta\ge Q^{-\sigma};\\
Q^{\varepsilon}\qquad&\text{when}\;\;\delta\le Q^{-\sigma}.
\end{aligned}
\right.
$$
Then by a similar covering argument and the Hausdorff-Cantelli lemma as in \S\ref{ProofJTC-Phi}, the desired conclusion follows on checking that each of the following series
$$\sum (2^k)^{-\tau m+(d+1)+\varepsilon-(1+\tau)s},
$$
$$
\sum (2^k)^{(d+1)-\frac{m}{\omega}+\varepsilon-(1+\tau)s},
$$
$$
\sum (2^k)^{-m\tau +m\sigma+\varepsilon-(1+\tau)s},
$$
and
$$
\sum(2^k)^{\varepsilon-(1+\tau)s}
$$
converges when $s$ is greater than $\frac{n+1}{\tau+1}-m$, $\frac{(d+1)\omega-m}{(\tau+1)\omega}$, $\frac{m(\sigma-\tau)}{\tau+1}$, and $0$ respectively, and choosing a sufficiently small $\varepsilon$.

\subsection{The divergence aspect}

\subsubsection{Ubiquitous systems}
We will use the ubiquitous system technique developed by Beresnevich, Dickinson, and Velani \cite{BDV06}. Let $\mathcal{U}$ be a compact set in $\R^d$ (viewed as a metric space under the supremum norm) and $\mathcal{R}:=\{R_j\}_{j\in \mathcal{J}}$ be a sequence of points in $\mathcal{U}$ (known as \emph{resonant points}) indexed by a countable set $\mathcal{J}$. Let $h:\mathcal{J}\to\N$ be a function that assigns a height $h_j$ to each point $R_j$. 
\begin{defn}
Let $\rho: \N\to\mathbb{R}_{>0}$ be a function such that $\lim_{q\to\infty}\rho(q)=0$. The system $(\mathcal{R},h)$ is called \emph{locally ubiquitous in $\mathcal{U}$ relative to $\rho$} if  for any ball $\mathcal{B}\subset\mathcal{U}$,
$$
\liminf_{k\to\infty}\mu_d\Bigg( \bigcup_{\substack{j\in\mathcal{J}\\h_j\le 2^k}}B\left(R_j,\rho(2^k)\right)\cap \mathcal{B}\Bigg)\ge\frac12\mu_d(\mathcal{B}).
$$
Here $B(\bm{x},r):=\{\bm{y}\in\R^d:|\bm{y}-\bm{x}|<r\}$. The function $\rho$ above is called the \emph{ubiquitous function}.
\end{defn}

For a function $\Psi:\N\to\R_{>0}$, let
\begin{align*}
\Lambda_\mathcal{R}(\Psi):&=\left\{\bm{x}\in\mathcal{U}:|\bm{x}-R_j|<\Psi(h_j)\text{ for infinitely many }j\in\mathcal{J}\right\}\\
&=\mathcal{U}\cap\limsup_{j\in\mathcal{J}}B(R_j,\Psi(h_j)).
\end{align*}
The following lemma is one of the main results on ubiquitous systems.
\begin{lem}\label{UbiquitousTheorem}
Let $(\mathcal{R},h)$ be a locally ubiquitous system in $\mathcal{U}$ relative to $\rho$, and let $\Psi:\N\to\R_{>0}$ be a decreasing function such that $\Psi(2^{k+1})\le \frac12\Psi(2^k)$ for all sufficiently large $k\in\N$. Then for any $s\in(0,d]$, we have
$$
\mathcal{H}^s(\Lambda_\mathcal{R}(\Psi))=\mathcal{H}^s(\mathcal{U})\quad \text{if}\quad \sum_{k=1}^\infty\frac{\Psi(2^k)^s}{\rho(2^k)^d}=\infty.
$$
\end{lem}

\subsubsection{Proof of Theorem \ref{JTD}}

We first of all verify that it can be assumed without loss of generality that
\begin{equation}\label{PsiBound}
 q^{-\frac{d+1-s}{m+s}-\varepsilon} \le  \psi(q)\le q^{-\frac{d+1-s}{m+s}}
\end{equation}
for any $\varepsilon>0$. Call the function on the left hand side $\eta(q)$ and the function on the right hand side $\xi(q)$. 
Consider the following two auxiliary functions
$$
\widehat{\psi}(q):=\max\{\psi(q),\eta(q)\},\quad \widetilde{\psi}(q):=\min\{\psi(q),\xi(q)\}.
$$
Clearly both $\widehat{\psi}$ and $\widetilde{\psi}$ are still decreasing functions. 

On one hand, note that
$$
\mathcal{S}_n(\widehat{\psi})=\mathcal{S}_n(\psi)\cup\mathcal{S}_n(\eta).
$$
Since $\eta$ satisfies the convergence condition \eqref{ConvCondition} in Theorem \ref{JTC} and the condition \eqref{RangeOfS-Div} implies \eqref{RangeOfS-Conv}, we may apply Theorem \ref{JTC} to obtain that
$$
\mathcal{H}^s(\mathcal{S}_n(\eta)\cap\mathcal{L})=0.
$$
So this shows that we may proceed with $\psi$ replaced by $\widetilde{\psi}$ and hence the first inequality in \eqref{PsiBound} can be assumed. 

On the other hand, observe that $\widetilde{\psi}$ satisfies the divergence condition \eqref{DivCondition}. 
Otherwise, if the series $\sum\widetilde{\psi}(q)^{m+s}q^{d-s}$ converges, then by the  Cauchy condensation we have $\widetilde{\psi}(q)^{m+s}q^{d+1-s}\to0$ as $q\to\infty$. This means that $\widetilde{\psi}(q)=o(\xi(q))$ and hence $\widetilde{\psi}(q)=\psi(q)$ for all but finitely many $q$. But this leads to the divergence of the above series about $\widetilde{\psi}$ since $\psi$ satisfies \eqref{DivCondition}, a contradiction. So in view of the inclusion $\mathcal{S}_n(\widetilde{\psi})\subseteq\mathcal{S}_n(\psi)$, it suffices to prove the theorem with $\psi$ replaced by $\widetilde{\psi}$, and therefore this justifies that the second inequality in \eqref{PsiBound} can be assumed.

Consider the sequence of resonant points in $\mathcal{U}=[0,1]^d$
$$
\mathcal{R}:=\left\{\bm{a}/q:(q,\bm{a})\in\mathcal{J}\right\},
$$
where
$$
\mathcal{J}:=\left\{(q,\bm{a})\in\N\times \Z^d:\|(q,\bm{a})A\|<\frac{\psi(q)}2,\,\frac{\bm{a}}q\in\mathcal{U}\right\}.
$$
For $(q,\bm{a})\in\mathcal{J}$, define $h_{(q,\bm{a})}:=q$.
Recall that $\Omega(A,\psi):=\Omega(A,\psi,\bm{0})$ is the projection of $\mathcal{S}_n(\psi)\cap\mathcal{L}$ onto $\R^d$ via the map $(\bm{x},\widetilde{\bm{x}}A)\to\bm{x}$.
\begin{lem}\label{LambdaInclusion} 
Let $\mathcal{R}$  be defined as above, and let
$$
\Psi(q):=\frac{c_2\psi(q)}q\enskip\text{with}\enskip c_2:=\min\left\{\frac1{2d|A'|},1\right\}.
$$
Then
$$
\Lambda_\mathcal{R}(\Psi)\subseteq\Omega(A,\psi).
$$
\end{lem}
\begin{proof}
Let $\bm{x}=(x_1,\ldots,x_d)\in\Lambda_\mathcal{R}(\Psi)$. Then
there are infinitely many $(q,\bm{a})\in\N\times \Z^d$ such that
$$
\left|\bm{x}-\frac{\bm{a}}q\right|<\Psi(q)=c_2\frac{\psi(q)}q\le \frac{\psi(q)}q
$$
and
$$
\left|(q,\bm{a})A-\bm{b}\right|<\frac{\psi(q)}2\quad\text{for some }\bm{b}\in\Z^m.
$$
Then by the triangle inequality
\begin{align*}
   \left|\widetilde{\bm{x}}A-\frac{\bm{b}}q\right|&\le \left|\widetilde{\bm{x}}A-\left(1,\frac{\bm{a}}q\right)A\right|+\left|\left(1,\frac{\bm{a}}q\right)A-\frac{\bm{b}}q\right|\\
   &< d|A'|\left|\bm{x}-\frac{\bm{a}}q\right|+\frac{\psi(q)}{2q}\\
   &\le(2d|A'|c_2+1) \frac{\psi(q)}{2q}\\
   &\le \frac{\psi(q)}{q}.
\end{align*}
Thus, we have shown that $(\bm{x},\widetilde{\bm{x}}A)\in\mathcal{S}_n(\psi)$, which by definition means that $\bm{x}\in\Omega(A,\psi)$.
\end{proof}

\begin{lem}\label{Ubiquity}

Let $\mathcal{R}$, $h$ be defined as above, and let $u:\N\to\R_{>0}$ be any function such that 
\begin{equation}\label{u(q)Condition}
\lim_{q\to\infty}u(q)=\infty\enskip\text{and}\enskip u(q)=O(q).
\end{equation}
With $\tau_0$ given in \eqref{tau0} and some $\varepsilon>0$, suppose that  
\begin{equation}\label{PsiTau}
1\ge\psi(q)\ge q^{-\tau_0+\varepsilon}.
\end{equation}
Then
$(\mathcal{R},h)$ is locally ubiquitous relative to $\rho$, where 
$$
\rho(q):=\left(\frac{2^mu(q)}{\psi(q)^mq^{d+1}}\right)^{\frac1d}.
$$
\end{lem}
\begin{proof}
First of all, from the conditions \eqref{u(q)Condition} and \eqref{PsiTau}, it is immediately clear that $\rho(q)\to0$ as $q\to\infty$.

Secondly, we know that there exists some $0<\kappa<1$ such that the conclusion of Theorem \ref{UbiquityLowerBound} holds with $\delta=\frac{\psi(Q)}2$. Moreover we may take $Q$ large enough so that $u(Q)>\kappa^{-d}$. Hence for sufficiently large $Q$ we have
\begin{align*}
&\mu_d\left(\mathcal{B}\cap \bigcup_{\substack{Q/u(Q)^{1/d}<q\le Q\\\bm{a}/q\in\mathcal{U}\\\|(q,\bm{a})A\|<\frac{\psi(Q)}2}}B\left(\frac{\bm{a}}q,\rho(Q)\right)\right)\\
\ge&\mu_d\left(\mathcal{B}\cap \bigcup_{(q,\bm{a})\in\mathcal{R}^\kappa(Q,\frac{\psi(Q)}2)}B\left(\frac{\bm{a}}q,\frac{\kappa^{-1}}{Q^{\frac{d+1}d}(\frac{\psi(Q)}2)^{\frac{m}d}}\right)\right)\\
\ge&\frac12\mu_d(\mathcal{B}),
\end{align*}
which demonstrates that $(\mathcal{R},h)$ is locally ubiquitous relative to $\rho$.
\end{proof}

\begin{lem}\label{FullMeasure}
With $\Psi$, $\mathcal{R}$, $h$, and $\rho$ defined as above, we have
$$
\mathcal{H}^s(\Lambda_\mathcal{R}(\Psi))=\mathcal{H}^s(\mathcal{U})\quad\text{if}\quad\sum_{q=1}^\infty \psi(q)^{m+s}q^{d-s}=\infty.
$$
\end{lem}
\begin{proof}
Since $\psi$ is decreasing, we have $\Psi(2^{k+1})\le \frac{\Psi(2^k)}2$. Due to \eqref{RangeOfS-Div}, we may find $\varepsilon>0$ so that
$$
\frac{d+1-s}{m+s}+2\varepsilon<\tau_0.
$$
This, together with \eqref{PsiBound}, implies in particular that \eqref{PsiTau} is satisfied. 

Now we construct $u(q)$ as follows. In view of the divergence condition \eqref{DivCondition}, we may find an increasing sequence of positive integers $\{q_i\}$ such that 
$$
\sum_{q_i\le q<q_{i+1}}\psi(q)^{m+s}q^{d-s}\ge 2^{d+1-s}.
$$
Then define $$u(q):=i\quad\text{for } q_{i}\le q< q_{i+1}.$$
Note that
\begin{align*}
    \sum_{q=1}^\infty\psi(q)^{m+s}q^{d-s}u(q)^{-1}&=\sum_{i=1}^\infty \sum_{q_i\le q<q_{i+1}}\psi(q)^{m+s}q^{d-s}u(q)^{-1}\\
    &\ge2^{d+1-s}\sum_{i=1}^\infty\frac{1}i=\infty.
\end{align*}
In addition, since $\psi(q)^{m+s}u(q)^{-1}$ is decreasing, we know by the Cauchy condensation that
\begin{equation}\label{UbiquityDiv}
\sum_{k=1}^\infty \psi(2^k)^{m+s}(2^k)^{d+1-s}u(2^k)^{-1}=\infty.
\end{equation}

Next we show that the function $u(q)$ defined above actually goes to infinity only at logarithmic rate. Since
\begin{align*}
    2^{d+1-s}&\le\sum_{q_i\le q<q_{i+1}}\psi(q)^{m+s}q^{d-s}\\
    &\le\sum_{q_i\le q<q_{i+1}}\psi(q_i)^{m+s}q_{i+1}^{d-s}\\
    &\overset{\eqref{PsiBound}}\le q_{i+1}^{d+1-s}q_{i}^{-(d+1-s)}
\end{align*}
we see that $q_{i+1}\ge2q_i$ and hence that
$$
q_i\ge2^{i-1}.
$$
This means that for $q_{i}\le q<q_{i+1}$, we have $\log_2q\ge\log_2q_i\ge i-1$. So
$$
u(q)=i\le \log_2q+1=O(q).
$$

Hitherto, we have verified the growth conditions \eqref{u(q)Condition} for $u(q)$ and \eqref{PsiTau} for $\psi(q)$. Therefore, by Lemma \ref{Ubiquity}, $(\mathcal{R},h)$ is locally ubiquitous relative to $\rho$. Moreover, it is readily seen that the series
\begin{align*}
    \sum_{k=1}^\infty\frac{\Psi(2^k)^s}{\rho(2^k)^d}&=\sum_{k=1}^\infty\frac{c_2^s\psi(2^k)^s(2^k)^{-s}}{2^mu(2^k)\psi(2^k)^{-m}(2^k)^{-(d+1)}}\\
    &=\frac{c_2^s}{2^m}\sum_{k=1}^\infty \frac{\psi(2^k)^{m+s}(2^k)^{d+1-s}}{u(2^k)}
\end{align*}
is divergent in view of \eqref{UbiquityDiv}. The proof of Lemma \ref{FullMeasure} is thus concluded after applying Lemma \ref{UbiquitousTheorem}.
\end{proof}

We are now ready to complete the proof of  Theorem \ref{JTD}. Under the range condition \eqref{RangeOfS-Div} and the divergence condition \eqref{DivCondition}, it follows from Lemma \ref{LambdaInclusion} and Lemma \ref{FullMeasure} that 
\begin{equation}\label{LocalFullMeasure}
\mathcal{H}^s(\Omega(A,\psi)\cap\mathcal{U})=\mathcal{H}^s(\mathcal{U}).
\end{equation}
When $s<d$, this already gives the desired conclusion $\mathcal{H}^s(\mathcal{S}_n(\psi)\cap\mathcal{L})=\infty$ in view of \eqref{Projection}. For the case $s=d$, we need to show that the complement of $\Omega(A,\psi)$ has zero $d$-dimensional Hausdorff measure in $\R^d$. To prove this, we first of all apply Lemma \ref{Shift} and obtain
$$
\mathcal{H}^d\left(\Omega(A, \psi)\cap \mathcal{U}_{\bm{k}}\right)=\mathcal{H}^d\left(\Omega(A_{\bm{k}},\psi)\cap \mathcal{U}\right),
$$
where $A_{\bm{k}}$ is defined in \eqref{Ak}. Furthermore, it is a consequence of Lemma \ref{Transform} that $\omega(A_{\bm{k}})=\omega(A)$.
So our argument shows that \eqref{LocalFullMeasure} still holds with $A$ replaced by $A_{\bm{k}}$. Hence $\Omega(A,\psi)\cap\mathcal{U}_{\bm{k}}$ has full measure in each $\mathcal{U}_{\bm{k}}$, which completes the proof.

\section{Further questions}
While we have solved some main problems regarding simultaneous diophantine approximation on affine subspaces,  below we compose a list of further questions related to the subject matter studied in this paper, not in any particular order. Some of them might be significantly harder than others. 
\begin{prob}
Close in the logarithmic gap in the classification of Khintchine type affine subspaces. Or even more ambitiously, try to formulate and prove a necessary and sufficient condition for an affine subspace to be of Khintchine type for convergence. 
\end{prob}

  \begin{prob}Can we give explicit affine subspaces that are of Khintchine type for convergence but not of strong Khintchine type for convergence? See Remark \ref{StrongKTC_VS_KTC}.
   \end{prob}

      \begin{prob}
      Prove that an affine subspace $\mathcal{L}_A$ is of Groshev type if $\omega(A)<n$. See Remark \ref{Groshev}.
      \end{prob}

      \begin{prob}
     For interesting objects, e.g. manifolds and fractals, is the Khintchine type theory always equivalent to the Groshev type theory? Put differently, is it true that as long as some object is of Khintchine type  then it is also of Groshev type  and vice versa? If not, can we find explicit counterexamples? Well, at least we can rule out nondegenerate manifolds and focus on affine subspaces and fractals. 
      \end{prob}

      \begin{prob}
      Is the condition $\omega(A) < \frac{m+ s}{d + 1 - s}$ in Theorem \ref{JTC} optimal up to the end point just like the  case $s=d$ on the Khintchine type theory? Precisely, is it true that if $\omega(A) > \frac{m+ s}{d + 1 - s}$ then the conclusion of Theorem \ref{JTC} is false? See Remark \ref{JTC-is-optimal?}.

      \end{prob}

        \begin{prob}
        Do nondegenerate submanifolds of an affine subspace inherit its Khintchine type property? In other words, is it true that an affine subspace is of Khintchine type if and only if all of its nondegenerate submanifolds are also of Khintchine type? The inheritance principle has been shown to hold for other diophantine properties such as extremality and strong extremality \cite{Kle03},    dual diophantine exponent \cite{Kle08}, and  simultaneous diophantine exponent  \cite{Zha09}. It is also partially known for the multiplicative diophantine exponent in the context of hyperplanes \cite{Zha12}.
        \end{prob}

\begin{prob}
Prove the inhomogeneous version of Theorem \ref{JTD}.
\end{prob}

\begin{prob}
Are all affine subspaces of Khintchine type for divergence? 
\end{prob}

\begin{prob}
Find an explicit formula for $\sigma(\mathcal{L})$ for arbitrary $\mathcal{L}$. See Remark \ref{SigmaFormula}.
\end{prob}

\begin{prob}
Find an exact formula for $\dim\mathcal{S}_n(\tau)\cap\mathcal{L}$  when $\frac1{\omega(A)}<\tau<\sigma(A)$. See Remark \ref{ExactHausdorffDim}.
\end{prob}

\proof[Acknowledgements]

The author expresses sincere gratitude to Professor John Friedlander for his invaluable guidance, generous support, and constant encouragement during the author’s early career, when it mattered the most. The author is also deeply grateful to his thesis advisor Professor Robert Vaughan for introducing him to the wonderland of metric diophantine approximation.

The author is indebted to the  anonymous referee for carefully reading the manuscript and providing helpful comments and suggestions.


\begin{thebibliography}{10}

\bibitem[Alv21]{Alv21}
D. Alvey,
\emph{A Khintchine-type theorem for affine subspaces}. 
Int. J. Number Theory 17 (2021), no. 6, 1323-1341.

\bibitem[BL07]{BL07}
D. Badziahin and J. Levesley, 
\emph{A note on simultaneous and multiplicative Diophantine approximation on planar curves}. 
Glasg. Math. J. 49 (2007), no. 2, 367-375.


\bibitem [Ber12]{Ber12} V. Beresnevich, \emph{Rational points near manifolds and metric Diophantine approximation}. Ann. of Math. (2) 175 (2012), no. 1, 187-235.

\bibitem[BBDD00]{BBDD00}
V. Beresnevich; V.I. Bernik; D. Dickinson; M. Dodson,
\emph{On linear manifolds for which the Khinchin approximation theorem holds. }
Vestsi Nats. Akad. Navuk Belarusi Ser. Fiz.-Mat. Navuk 2000, no. 2, 14-17, 139.

\bibitem[BDV01]{BDV01}
V. Beresnevich; D. Dickinson; S. Velani, 
\emph{Sets of exact `logarithmic' order in the theory of Diophantine approximation}. 
Math. Ann. 321 (2001), no. 2, 253-273.

\bibitem[BDV06]{BDV06}
V. Beresnevich; D. Dickinson; S. Velani, \emph{Measure theoretic laws for lim sup sets.} Mem. Amer. Math. Soc. 179 (2006), no. 846, x+91 pp.

\bibitem[BDV07]{BDV07} V. Beresnevich; D. Dickinson; S. Velani,   \emph{Diophantine approximation on planar curves and the distribution of rational points.} With an Appendix II by R. C. Vaughan. Ann. of Math. (2)  166  (2007),  no. 2, 367-426.

\bibitem[BGGV20]{BGGV20} 
V. Beresnevich; A. Ganguly; A. Ghosh; S. Velani, \emph{Inhomogeneous dual Diophantine approximation on affine subspaces}. Int. Math. Res. Not. IMRN 2020, no. 12, 3582-3613.

\bibitem[BHV20]{BHV20}
V. Beresnevich; A. Haynes; S. Velani, 
\emph{Sums of reciprocals of fractional parts and multiplicative Diophantine approximation}. 
Mem. Amer. Math. Soc. 263 (2020), no. 1276, vii + 77 pp.

\bibitem[BLVV17]{BLVV17}
V. Beresnevich; L. Lee; R.C. Vaughan; S. Velani,
\emph{Diophantine approximation on manifolds and lower bounds for Hausdorff dimension}. 
Mathematika 63 (2017), no. 3, 762-779.


\bibitem[BVV11]{BVV11}

 V. Beresnevich; R.C. Vaughan; S. Velani, \emph{Inhomogeneous Diophantine approximation on planar curves}. Math. Ann. 349 (2011), no. 4, 929-942. 


\bibitem[BVVZ17]{BVVZ17}
V. Beresnevich; R.C. Vaughan; S. Velani; E. Zorin, \emph{Diophantine approximation on manifolds and the distribution of rational points: contributions to the convergence theory}. Int. Math. Res. Not. IMRN 2017, no. 10, 2885-2908.

\bibitem[BVVZ21]{BVVZ21}
V. Beresnevich; R.C. Vaughan; S. Velani; E. Zorin, \emph{Diophantine approximation on curves and the distribution of rational points: Contributions to the divergence theory}. Adv. Math. 388 (2021), Paper No. 107861, 33 pp.




\bibitem[BV06]{BV06}
V. Beresnevich and S. Velani, 
\emph{A mass transference principle and the Duffin-Schaeffer conjecture for Hausdorff measures}.
Ann. of Math. (2) 164 (2006), no. 3, 971-992.

\bibitem[BV10]{BV10}
V. Beresnevich and S. Velani, 
\emph{An inhomogeneous transference principle and Diophantine approximation.}
Proc. Lond. Math. Soc. (3) 101 (2010), no. 3, 821-851.

\bibitem[BV15]{BV15}
V. Beresnevich and S. Velani, 
\emph{A note on three problems in metric Diophantine approximation}. Recent trends in ergodic theory and dynamical systems, 211-229,
Contemp. Math., 631, Amer. Math. Soc., Providence, RI, 2015.

\bibitem[BY23]{BY23}
V. Beresnevich and L. Yang,
\emph{Khintchine's theorem and Diophantine approximation on manifolds}. Acta Math. 231 (2023), no. 1, 1-30.


\bibitem[BD99]{BD99}
V.I. Bernik and M.M. Dodson, \emph{Metric Diophantine approximation on manifolds.} Cambridge Tracts in Mathematics, 137. Cambridge University Press, Cambridge, 1999. 

\bibitem[BD86]{BD86}
J.D. Bovey and M.M. Dodson, 
\emph{The Hausdorff dimension of systems of linear forms}.
Acta Arith. 45 (1986), no. 4, 337-358.

\bibitem[BL05]{BL05}
Y. Bugeaud and M. Laurent,
\emph{On exponents of homogeneous and inhomogeneous Diophantine approximation}. 
Mosc. Math. J. 5 (2005), no. 4, 747-766, 972.

\bibitem[Cas57]{Cas57}
J. W. S. Cassels,
\emph{An introduction to Diophantine approximation}.
Cambridge Tracts in Mathematics and Mathematical Physics, No. 45. Cambridge University Press, New York, 1957.



\bibitem[CY]{CY}
S. Chow and L. Yang, 
\emph{Effective equidistribution for multiplicative Diophantine approximation on lines}. Invent. Math., to appear. Preprint available at 	\href{https://arxiv.org/pdf/1902.06081.pdf}{arXiv:1902.06081}.



\bibitem[DRV91]{DRV91}
M.M. Dodson, B.P. Rynne, and J.A.G. Vickers,
\emph{Khintchine-type theorems on manifolds}.
Acta Arith. 57 (1991), no. 2, 115-130.


\bibitem[Gho11]{Gho11}
A. Ghosh, \emph{Diophantine exponents and the Khintchine Groshev theorem}. 
Monatsh. Math. 163 (2011), no. 3, 281-299.

\bibitem[HL22]{HL22}
G.H. Hardy and J.E. Littlewood, 
\emph{Some Problems of Diophantine Approximation: The Lattice-Points of a Right-Angled Triangle}.
Proc. London Math. Soc. (2) 20 (1922), no. 1, 15-36.

\bibitem[Hua15]{Hua15}
J.-J. Huang, \emph{Rational points near planar curves and Diophantine approximation}.
Adv. Math. 274 (2015), 490-515.



\bibitem[Hua19]{Hua19}
J.-J. Huang, \emph{Integral points close to a space curve}. Math. Ann. 374 (2019), no. 3-4, 1987-2003. 

\bibitem[Hua20a]{Hua20a}
J.-J. Huang, \emph{Diophantine approximation on the parabola with non-monotonic approximation functions}. Math. Proc. Cambridge Philos. Soc. 168 (2020), no. 3, 535-542.

\bibitem[Hua20b]{Hua20b}
J.-J. Huang, \emph{The density of rational points near hypersurfaces}.
Duke Math. J. 169 (2020), no. 11, 2045-2077. 




\bibitem[HL21]{HL21}
J.-J. Huang and J. Liu,
\emph{Simultaneous approximation on affine subspaces}. 
Int. Math. Res. Not. IMRN 2021, no. 19, 14905-14921.



\bibitem[Hux68]{Hux68}
M.N. Huxley,
\emph{The large sieve inequality for algebraic number fields}.
Mathematika 15 (1968), 178-187.

\bibitem[Khi26]{Khi26}
A. Khintchine, \emph{Zur metrischen Theorie der diophantischen Approximationen.} Math. Z. 24 (1926), no. 1, 706-714.



\bibitem[Kle03]{Kle03}
D. Kleinbock, 
\emph{Extremal subspaces and their submanifolds.}
Geom. Funct. Anal. 13 (2003), no. 2, 437-466. 

\bibitem[Kle08]{Kle08}
D. Kleinbock, \emph{An extension of quantitative nondivergence and applications to Diophantine exponents.}
Trans. Amer. Math. Soc. 360 (2008), no. 12, 6497-6523.

\bibitem[KM98]{KM98}
D.Y. Kleinbock and G.A. Margulis,  
\emph{Flows on homogeneous spaces and Diophantine approximation on manifolds}. 
Ann. of Math. (2) 148 (1998), no. 1, 339-360.

\bibitem[Kov00]{Kov00}
\`{E}.I. Kovalevskaya, \emph{On the exact order of simultaneous approximations of almost all points of linear manifolds.} Vests\={\i} Nats. Akad. Navuk Belarus\={\i} Ser. F\=\i z.-Mat. Navuk 2000, no. 1, 23-27, 140.

\bibitem[Mon78]{Mon78}
H.L. Montgomery, 
\emph{The analytic principle of the large sieve}.
Bull. Amer. Math. Soc. 84 (1978), no. 4, 547-567.

\bibitem[Mun]{Mun}
F. Munkelt,
\emph{On the number of rational points close to a compact manifold under a less restrictive curvature condition}.
Preprint available at \href{https://arxiv.org/pdf/2205.06183.pdf}{arXiv:2205.06183}.




\bibitem[Ram15]{Ram15}
F. Ram\'{i}rez, \emph{Khintchine types of translated coordinate hyperplanes.} Acta Arithmetica 170 (2015), no. 3, 243-273.

\bibitem[RSS17]{RSS17}
F. Ram\'{i}rez; D. Simmons; F. S\"{u}ess, \emph{Rational approximation of affine coordinate subspaces of Euclidean space.} Acta Arith. 177 (2017), no. 1, 91-100. 

\bibitem[SY22]{SY22}
D. Schindler and S. Yamagishi, 
\emph{Density of rational points near/on compact manifolds with certain curvature conditions}. 
Adv. Math. 403 (2022), Paper No. 108358, 36 pp.

\bibitem[Sch64]{Sch64}
W.M. Schmidt, \emph{Metrische S\"atze \"uber simultane Approximation abh\"angiger Gr\"ossen}. Monatsh. Math. 68 (1964), 154-166.

\bibitem[SY]{SY}
N. Shah and P. Yang, \emph{Equidistribution of expanding degenerate manifolds in the space of lattices}. Preprint available at \href{https://arxiv.org/pdf/2112.13952.pdf}{arXiv:2112.13952}.


\bibitem[Sim18]{Sim18}
D. Simmons, \emph{Some manifolds of Khinchin type for convergence.} J. Th\'{e}or. Nombres Bordeaux 30 (2018), no. 1, 175-193.

\bibitem[Spr69]{Spr69}
V.G. Sprind\v{z}uk, 
\emph{Mahler's problem in metric number theory}.
Translations of Mathematical Monographs, 25. American Mathematical Society, Providence, RI, 1969 vii+192 pp. 


\bibitem[Spr79]{Spr79}
V.G. Sprind\v{z}uk, \emph{Metric theory of Diophantine approximations.}  John Wiley \& Sons, 1979.


\bibitem [VV06]{VV06}
R.C. Vaughan and S. Velani, \emph{Diophantine approximation on planar curves: the convergence theory.} Invent. Math.  166  (2006),  no. 1, 103-124.

\bibitem[Zha09]{Zha09}
Y. Zhang, 
\emph{Diophantine exponents of affine subspaces: the simultaneous approximation case}. 
J. Number Theory 129 (2009), no. 8, 1976-1989.


\bibitem[Zha12]{Zha12}
Y. Zhang, 
\emph{Multiplicative Diophantine exponents of hyperplanes and their nondegenerate manifolds}. 
J. Reine Angew. Math. 664 (2012), 93-113.

\end{thebibliography}
\end{document}